\documentclass[a4paper]{amsart}
\usepackage{amsmath,amssymb,amsthm}
\usepackage[british]{babel}
\usepackage{eucal}
\usepackage{tikz}
\usepackage{subcaption}
\usepackage{booktabs}
\newcommand{\ddiv}{\operatorname{div}}
\newcommand{\curl}{\operatorname{curl}}
\newcommand{\Div}{\operatorname{Div}}
\newcommand{\dDiv}{\operatorname{divDiv}}
\newcommand{\dom}{G}

\newtheorem{theorem}{Theorem}
\newtheorem{lemma}{Lemma}
\theoremstyle{definition}
\newtheorem{definition}[lemma]{Definition}
\newtheorem{corollary}[lemma]{Corollary}
\newtheorem{assumption}{Assumption}
\newtheorem{example}[lemma]{Example}
\theoremstyle{remark}
\newtheorem{remark}[lemma]{Remark}

\numberwithin{equation}{section}
\numberwithin{lemma}{section}

\date{}

\begin{document}
	%--------------------------------------------------------------------
	% Title, Author, etc.
	%--------------------------------------------------------------------
	\title[Transposed saddle-point problems]{The transposition method in saddle-point problems}
	\author{Dietmar Gallistl}
	\thanks{D.~Gallistl is supported by the European Research Council
		(StG \emph{DAFNE}, ID 891734).}
	\address{Institut f\"ur Mathematik, Universit\"at Jena, 07743 Jena, Germany}
	\email{dietmar.gallistl [at] uni-jena.de ; z.liu [at] uni-jena.de}
	\author{Zhen Liu}
	\thanks{Z.~Liu is supported by the Sino-German (CSC-DAAD) Postdoc Scholarship Program, ID 202406010516.}
% 	\address{Institut f\"ur Mathematik, Universit\"at Jena, 07743 Jena, Germany}
% 	\email{}
	%--------------------------------------------------------------------

\begin{abstract}
The transposition method known from the theory of second-order elliptic
partial differential equations is formulated for saddle-point problems
under certain structural assumptions known from the Brezzi splitting
plus mapping properties that reflect elliptic regularity in concrete
instances of differential operators with boundary conditions.
A framework is devised that grants approximation by standard,
unmodified Galerkin-type methods,
which may be nonconforming with respect to the more general data
allowed by the transposed operator.
That the data need not be regularized is an advantage particular
to the saddle-point approach.
Examples covered by the framework include the Laplacian, the Stokes operator,
linear elasticity, 
the bi-Laplacian, and the stationary Maxwell problem
with $L^2$ boundary data.
\end{abstract}
	

	\maketitle
	
	\section{Introduction}
	
	The transposition method is an elementary mechanism that converts
	\emph{regularity of solutions} to an elliptic partial differential
	equation (PDE) for certain classes
	of right-hand sides to the \emph{regularity of test functions}
	of the adjoint (transposed) operator.
	It was described in \cite{LionsMagenes1972} and
	\cite{Grisvard1985,GoulaouicGrisvard1970} for second-order
	elliptic operators with boundary data in $L^2$, which do not
	admit solutions with bounded Dirichlet energy.
	The numerical approximation by finite element methods (FEM)
	was studied by
	\cite{Berggren2004,ApelNicaisePfefferer2016,
		ApelNicaisePfefferer2017,Duran2020,Gao2025}. 
	The use of primal methods requires a fit of the boundary
	data to the trace space because boundary conditions are
	imposed in an essential fashion.
	This is not necessary when mixed dual methods are used because
	the Dirichlet datum appears on the right-hand side only.
	The compatibility issue is then shifted to the fact that
	the data may not be compatible with the mapping properties
	of the trace.
	As an effect, the discretized saddle-point problem may then be
	well-posed without further regularization: a decisive advantage
	of the mixed method. Its convergence as an approximation scheme,
	however, deserves justification. In particular, the mixed method,
	though not modified compared to its original form, may become
	nonconforming when applied to the transposed problem because
	the discrete functions may not form a subspace of the smoother
	test functions of the transposed problem.
	As an illustrating example, we mention Poisson's problem with
	$L^2$ boundary data, where the very weak solution is defined by
	testing with functions from the space $H^{1+s}(\Omega)$
	for some $s>1/2$. Their gradients then belong to $H^s$,
	but the usual discrete Raviart--Thomas vector fields 
	only belong to $H^\nu$ with $\nu<1/2$.
	In \cite{Gao2025}, the mixed Raviart--Thomas method was
	justified for the Laplacian boundary value problem with $L^2$
	data, where the convergence analysis used the regularization
	framework of \cite{ApelNicaisePfefferer2016}.
	In this work we give an alternative and to some extent more
	structure-oriented approach that circumvents any regularization
	in the method and its error analysis.
	
	The arguments used here are fairly generic in the sense that
	they are structural on the level of functional analysis of
	saddle-point problems,
	in spite of elliptic PDE being the principal application.
	We therefore propose a framework that departs from a stable
	saddle-point problem in the sense of Brezzi
	\cite{BoffiBrezziFortin2013} in Hilbert spaces
	$\Sigma\times U$,
	where $\Sigma$ is a subspace of a possibly larger
	Hilbert space $H$ with $\Sigma\subseteq H$.
	Given data $g\in \Sigma'$ and $f\in U'$, the saddle-point
	problem seeks $(\sigma,u)\in \Sigma\times U$ such that
	\begin{align*}
		\begin{aligned}
			&a(\sigma,\tau ) + b(\tau,u) &=&& \langle g ,\tau\rangle
			\\
			&b(\sigma,v)                 &=&& -\langle f,v\rangle 
		\end{aligned}
		\qquad\text{for all }  (\tau,v)\in\Sigma\times U.
	\end{align*}
	Assuming standard conditions on $a$, $b$ from the Brezzi
	splitting, this problem is well posed,
	and boundary data of $u$ in standard PDE applications
	will describe subclasses of $\Sigma'$.
	The well known idea of transposition \cite{LionsMagenes1972}
	is to use elliptic regularity and a transposed isomorphism,
	which results in an equation with more regular test functions,
	thereby giving a meaning to less regular data.
	Since the mixed formulation of elliptic problems shifts
	the boundary data from an essential condition to a datum
	on the right-hand side, we can make use of the fact that
	the discrete equations are well-defined whenever the discrete
	right-hand side is. In this work we aim at an error analysis
	of the resulting approximations under generic assumptions
	on mixed discretizations,
	which is Theorem~\ref{thm:error}.
	It states an error estimate for the approximation of
	$u$ in the norm of $U$, which in the PDE context this will often
	be the $L^2$ norm.
	We give a series of instances in particular PDEs to illustrate
	the implication to actual methods.
	
	The remaining parts of this article are organized as follows.
	Section~\ref{sec:abstract} presents, as our main result,
	an error estimate for the mixed Galerkin-type methods
	for the problem with irregular data $g$.
	Applications to specific model problems are then detailed
	in the subsequent sections, including the Laplacian (Section~\ref{sec:Laplacian}), the Stokes problem (Section~\ref{sec:Stokes}), linear elasticity (Section~\ref{sec:Elasticity}), the biharmonic equation (Section~\ref{sec:Biharmonic}), and the Maxwell equation (Section~\ref{sec:Maxwell}).
	
	\bigskip
	Throughout this work,
        an inequality $A\leq C B$ with a constant $C$ that is independent
	of discretization parameters is denoted by $A\lesssim B$.
	
	\section{Framework, discretization, main result}
	\label{sec:abstract}
	
	This section develops the transposition method for 
	saddle-point problems and analyzes a class of possibly nonconforming 
	mixed discretizations together with error estimates.
	
	\subsection{The continuous mixed problem}
	
	We start with a basic structural setting of saddle-point problems.
	\begin{assumption}[structure]
		\label{ass:saddle}
		(i) Let $\Sigma$, $U$ be Hilbert spaces 
		(the corresponding norms are denoted by
		$\|\cdot\|_\Sigma$ and $\|\cdot\|_U$)
		with bounded bilinear forms
		$$
		a:\Sigma\times\Sigma\to \mathbb R,
		\quad
		b:\Sigma\times U\to \mathbb R.
		$$
		Let $H$ be a Hilbert space with the norm $\| \cdot \|_H$ such
		that $\Sigma\hookrightarrow H$ continuously. 
		We assume that $a$ extends to a bounded, symmetric, positive semidefinite
		form on $H\times H$. Thus, 	
		$|a(\sigma,\tau)| \lesssim \| \sigma\|_H \|\tau\|_H$
		for all $\sigma,\tau\in H$.
		
		(ii) The Brezzi conditions hold: inf-sup of $b$ and $\Sigma$-coercivity
		of $a$ on the kernel of $b$,
		that is, there exist constants $\alpha_0>0$ and $\beta_0>0$
		such that
			$$
			\inf_{0\ne \tau\in Z}
			 \|\tau\|_{\Sigma}^{-2}a(\tau,\tau)\ge\alpha_0 
			\quad\text{and}\quad
			\inf_{0\ne v\in U}\sup_{0\ne\tau\in\Sigma}
                       ( \|\tau\|_{\Sigma}\|v\|_U)^{-1}
                       b(\tau,v)\ge\beta_0,
			$$
			where $Z:= \{\tau \in\Sigma: b(\tau,v)=0 \text{ for all } v\in U\}$
			denotes the kernel of $b$.
		\qed
	\end{assumption}
	
	The coercivity of $a$ on the kernel of $b$, 
	together with the inf-sup condition for $b$
	are sufficient for the well-posedness of the mixed problem
	\cite{BoffiBrezziFortin2013}.
	
	We identify  $U'$ with $U$ through the Riesz isomorphism of $U$
	in this paper. Let  $\langle \cdot, \cdot \rangle$ denote the duality pairing of $\Sigma' \times \Sigma$. 
	For every $g \in\Sigma'$ and $f\in U'$, consider the mixed problem: find $(\sigma, u) \in \Sigma \times U$ such that
	\begin{equation}
		\label{eq:classical-saddle}
		\left\{ 
		\begin{aligned}
			a(\sigma,\tau)+b(\tau,u)&=\langle g,\tau\rangle &&\text{ for all } \tau\in\Sigma,
			\\
			b(\sigma,v)&= - (f,v)_U &&\text{ for all } v\in U.
		\end{aligned} \right.
	\end{equation}
	Assumption \ref{ass:saddle} and the saddle-point theory \cite{BoffiBrezziFortin2013} 
	imply that \eqref{eq:classical-saddle} has a unique solution and
	\begin{equation}\label{e:stability_standard}
		\|\sigma\|_\Sigma+\|u\|_U \lesssim
		\|g\|_{\Sigma'}+\|f\|_{U}.
	\end{equation}
	Accordingly, we define the continuous saddle-point operator 
	$$
	T:\Sigma\times U \to \Sigma'\times U
	$$
	by $T(\sigma,u)=(g,f)$ if and only if $(\sigma, u)$ satisfies \eqref{eq:classical-saddle}.
	By Assumption \ref{ass:saddle} and \eqref{e:stability_standard},
	$T$ is an isomorphism.

	\subsection{The transposition method for the mixed formulation}
	
	We are interested in irregular boundary data $g \not \in \Sigma'$.
	In this setting, the original mixed formulation is not defined
	and the transposition method for the mixed formulation is invoked.
	
	\begin{assumption}[elliptic regularity]
		\label{ass:regularity}
		Assume we are given Banach subspaces
		$\Sigma^r\subseteq\Sigma$ and $U^r \subseteq U$ endowed with 
		stronger norms $\| \cdot \|_{\Sigma^r}$ and $\| \cdot \|_{U^r}$, 
		respectively. For given $z \in U$, any solution to $T(\phi, w) = (0,z)$
		satisfies
		\begin{equation}\label{e:regularity}
		 \| \phi\|_{\Sigma^r} + \| w \|_{U^r} \lesssim \| z\|_{U}
		 .
		\end{equation}
		\qed
	\end{assumption}
	
	By Assumption \ref{ass:regularity}, we can define the transposition solution as follows.
	
	\begin{definition}
		Let Assumptions~\ref{ass:saddle}--\ref{ass:regularity} hold.
		Let $f\in U$ and let $g \in (\Sigma^r)'$.  An element $u\in U$ is 
		called the transposition solution associated with data $(f,g)$, if
		\begin{equation}
			\label{def:transposed}
			(u,z)_U
			=(f,w)_U - \langle g,\phi \rangle
			\qquad\text{ for all } z\in U,
		\end{equation}
		where $(\phi, w)$ is determined by $T(\phi, w) = (0,z)$.  
		For any $z\in U$,
		the extended boundary value of the transposition solution $u \in U$ is characterized by
		$$
		\langle\gamma u,\phi\rangle := (f,w)_U-(u,z)_U,
		$$
		for every pair $(\phi,w)$ satisfying $T(\phi,w)=(0,z)$.
	\end{definition}

	\begin{lemma}[transposition solution]
		\label{lem:continuous}
		Assume that Assumptions~\ref{ass:saddle} and~\ref{ass:regularity} hold.
		For every $f \in U$ and $g \in (\Sigma^r)'$, 
		there exists a unique transposition solution $u\in U$ to
		\eqref{def:transposed}.
		If $g \in \Sigma'$, then the transposition solution coincides with
		the solution of the classical mixed formulation $T(\sigma,u)=(g,f)$ for some $\sigma\in\Sigma$.
	\end{lemma}
	
	\begin{proof}
	    Given $z \in U$,
		Assumption~\ref{ass:regularity} implies that there exists a unique solution 
		$(\phi, w) \in \Sigma^r \times U^r$ to $T(\phi, w) = (0,z)$.
		Define the linear functional
		$\mathcal F(z) := (f,w)_U-\langle g,\phi\rangle$ for $z \in U$. 
		The regularity estimate \eqref{e:regularity}
		and $U^r\hookrightarrow U$ continuously show
		$$
		|\mathcal F(z)| \leq \|g\|_{(\Sigma^r)'} \|\phi\|_{\Sigma^r}
		+ \|f\|_U\|w\|_U \lesssim \left( \|g\|_{(\Sigma^r)'} + \|f\|_U
		\right)\|z\|_U.
		$$
		Thus $\mathcal F$ is continuous on $U$.
		By the Riesz isomorphism,
		there exists a unique $u\in U$ such that $(u,z)_U=\mathcal F(z)$ for all $z\in U$. Moreover,
		$$
		\|u\|_U = \|\mathcal F\|_{U'} \lesssim \|f\|_U+\|g\|_{(\Sigma^r)'}.
		$$
		This proves the existence, uniqueness and stability of the transposition solution.
		It remains to show that for $g\in \Sigma'$ this reproduces
		solutions to the original problem.
		For $g\in\Sigma'$, Assumption~\ref{ass:saddle} 
		shows that there exists a unique $(\sigma,u_c)\in\Sigma\times U$ satisfying $T(\sigma,u_c)=(g,f)$.
		Taking $\tau=\phi$ in the first equation of $T(\sigma,u_c)=(g,f)$ gives
		$\langle g,\phi\rangle = a(\sigma,\phi)+b(\phi,u_c)$. 
		Choosing $\tau=\sigma$ in the first equation of $T(\phi, w) = (0,z)$ and using the symmetry of $a(\cdot,\cdot)$ yields $a(\sigma,\phi)+b(\sigma,w)=0$.
		The second equation of $T(\sigma,u_c)=(g,f)$ then implies $(f,w)_U= a(\sigma,\phi)$.
		Consequently,
		$$
		(f,w)_U-\langle g,\phi\rangle=a(\sigma,\phi)-\bigl(a(\sigma,\phi)+b(\phi,u_c) \bigr) =-b(\phi,u_c).
		$$
		Using the second equation of $T(\phi, w) = (0,z)$ with $v=u_c$ gives $b(\phi,u_c)=-(z,u_c)_U$. Hence,
		$$
		(f,w)_U-\langle g,\phi\rangle = (z,u_c)_U = (u_c,z)_U \quad \text{ for all } z \in U.
		$$
		Therefore, $u_c$ satisfies the defining identity of the transposition solution. The uniqueness of the transposition solution completes the proof.
	\end{proof}
	
	\begin{remark}\label{rem:iso}
	 The transposition solution from Lemma~\ref{lem:continuous}
	 satisfies $\langle \gamma u,\phi\rangle = \langle g,\phi\rangle$
	 only for solutions to $T(\phi,w)=(0,z)$ with $z\in U$.
	 The relation $\gamma u = g$ as an identity in $(\Sigma^r)'$
	 would require the additional condition
	 that $T$ is an isomorphism between $\Sigma^r\times U^r$ and $\{0\}\times U$.
	\end{remark}

	\subsection{Mixed discretization}
	
	Let families of finite-dimensional spaces $\Sigma_h\subset H$ and $U_h\subset U$,
	where $h$ is a formal parameter that can be thought of as a discretization 
	length-scale in practice.
	Let $\| \cdot \|_{\Sigma,h}$ be a seminorm on $\Sigma_h+\Sigma$,
	which coincides with $\|\cdot\|_\Sigma$ for elements of $\Sigma$
	and which is a norm on $\Sigma_h$.
	We do not assume $\Sigma_h \subset \Sigma$ here,
	but assume that the embedding $\Sigma_h\hookrightarrow H$ is uniform
	in $h$.
	Assume that we are given a bilinear form $b_h$ on 
	$(\Sigma_h+\Sigma)\times U$ with
	$b_h(\tau,\cdot)=b(\tau,\cdot)$ for all $\tau\in\Sigma$ 
	and $|b_h(\tau_h, v)| \lesssim \|\tau_h \|_{\Sigma,h} \| v\|_{U}$ for 
	all $\tau_h \in \Sigma_h$.
		
	We recall that both $\Sigma^r$ and $\Sigma_h$ are subspaces of $H$,
	and so is their sum. 
	Abusing notation, given a linear functional on $\Sigma_h+\Sigma^r$,
	we write $g \in \Sigma_h' \cap (\Sigma^r)'$ if
	$\|g\|_{\Sigma_h'}$ and $\|g\|_{(\Sigma^r)'}$ are finite,
	the norms being the usual operator norms with respect to
	$\|\cdot\|_{\Sigma,h}$ and $\|\cdot\|_{\Sigma^r}$, respectively.
	The space of such $g$ is denoted by
	$\Sigma_h' \cap (\Sigma^r)'$.
		
	The discrete mixed problem reads:
	Given $g \in \Sigma_h' \cap (\Sigma^r)'$ and $f \in U$,
	find $(\sigma_h,u_h)\in\Sigma_h\times U_h$ such that
	\begin{equation}
		\label{eq:discrete}
		\left\{  
		\begin{aligned}
			a(\sigma_h,\tau_h)+b_h(\tau_h,u_h)&=\langle g,\tau_h\rangle
			&&\text{ for all } \tau_h\in\Sigma_h, \\
			b_h(\sigma_h,v_h)&=-(f,v_h)_U
			&&\text{ for all } v_h\in U_h.
		\end{aligned} \right.
	\end{equation}
	
	\begin{assumption}[discrete stability]
		\label{ass:discrete-stability}
		The discrete formulation \eqref{eq:discrete} is well-posed:
		There exist constants $\alpha_1>0$ and $\beta_1>0$,
		independent of $h$, such that
		$$
		\inf_{0\ne \tau_h\in Z_h}
                \|\tau_h\|_{\Sigma,h}^{-2}
		a(\tau_h,\tau_h)\ge\alpha_1,
		\quad
		\inf_{0\ne v_h\in U_h}\sup_{0\ne\tau_h\in\Sigma_h}
		(\|\tau_h\|_{\Sigma,h}\|v_h\|_U)^{-1} b_h(\tau_h,v_h)\ge\beta_1
		$$
		with $Z_h:= \{\tau_h\in\Sigma_h:b_h(\tau_h,v_h)=0 \text{ for all } v_h\in U_h \}$.
	\end{assumption}
	Assumption~\ref{ass:discrete-stability} as a standard condition
        in  saddle-point theory
	\cite{BoffiBrezziFortin2013} implies that \eqref{eq:discrete} 
	has a unique solution and that
	\begin{equation}
		\label{eq:discrete-stability}
		\|\sigma_h\|_{\Sigma,h}+\|u_h\|_U \lesssim \|g\|_{\Sigma_h'}+\|f\|_{U}
         .
	\end{equation}
	In what follows, three further quantitative
        assumptions on the discretization
	spaces are given for the sake of proving error estimates of the
	discrete solution $(\sigma_h, u_h)$ as an approximation to the transposition solution $u$.
	
	\begin{assumption}[uniform approximation of $U^r$]
		\label{condition1}
		Assume there exists a linear continuous operator
		$P_h: U^r\to U_h$ and a function 
		$\rho_U=\rho_U(h)$ with
		$\rho_U\to0$ as $h\to0$ satisfying
		$$
		\|w-P_hw\|_U \lesssim\rho_U\|w\|_{U^r} \quad \text{ for all } w\in U^r.
		$$
	\end{assumption}
	
	\begin{assumption}[commuting approximation of $\Sigma^r$]
		\label{condition2}
		Assume there exists a linear and continuous operator
		$\Pi_h:\Sigma^r\to\Sigma_h$ with the following properties.
		Given a discrete $z_h \in U_h$,
		we write $\phi=\phi(z_h)$ for the first component of the solution
		$(\phi, w) \in \Sigma^r \times U^r$ 
		of $T(\phi, w) = (0, z_h)$.
		Assume that there are functions
		$\rho_\Sigma=\rho_\Sigma(h)$
		and $\rho_X=\rho_X(h)$ with
		$\rho_\Sigma,\rho_X\to0$ as $h\to 0$
		and a Banach space
		$X \subset \Sigma_h' \cap (\Sigma^r)'$ such that
		\begin{align}
			b_h(\Pi_h\phi,v_h)&= b(\phi, v_h)
			&&\text{ for all } z_h \in U_h,  v_h\in U_h, \tag{i} \\
			\|\phi-\Pi_h\phi\|_H &\lesssim \rho_\Sigma\|\phi\|_{\Sigma^r}
			&&\text{ for all }z_h \in U_h, \tag{ii}	\\
			|\langle g,\phi-\Pi_h\phi\rangle| &\lesssim \rho_X\|g\|_X\|\phi\|_{\Sigma^r}
			&&\text{ for all } z_h \in U_h,g\in X. \tag{iii}
		\end{align}
	\end{assumption}
	
	\begin{assumption}[nonconformity control]
		\label{condition3}
		Given $ \phi \in \Sigma^r$ and $w \in U^r$ satisfying
		$a(\phi, \tau) + b(\tau, w)=0$ for all $\tau \in \Sigma$,
		we define the nonconforming consistency residual by
		$$
		R_h(\tau_h;\phi,w)
		:=a(\tau_h,\phi)+b_h(\tau_h,w) \quad \text{ for all } \tau_h\in\Sigma_h
		$$
		and assume that there exists a function 
		$\rho_{\mathrm{nc}}=\rho_{\mathrm{nc}}(h)$ with
		$\rho_{\mathrm{nc}}\to0$ as $h\to0 $ such that 
		$$
		|R_h(\tau_h;\phi,w)|\lesssim\rho_{\mathrm{nc}}\|\tau_h\|_{\Sigma,h}\left(\|\phi\|_{\Sigma^r}+\|w\|_{U^r}\right)
		\quad\text{for all }\tau_h\in\Sigma_h
		$$
		and all such pairs $(\phi,w)$.
	\end{assumption}
	
	\subsection{Error estimate}
	
	As our main theorem we give an abstract error estimate 
        based on the continuous and discrete mixed formulations.
	
	\begin{theorem}[error estimate]
		\label{thm:error}
		Let Assumptions~\ref{ass:saddle}--\ref{condition3}
		be satisfied.
		Given $f \in U$ and $g \in X$,
		let $u \in U$ be the transposition solution defined in 
		\eqref{def:transposed} and $(\sigma_h, u_h)\in \Sigma_h \times U_h$ 
		be the solution of \eqref{eq:discrete}. 
		Then the following error estimate holds
		$$
		\begin{aligned}
			\|u-u_h\|_U & 
			\lesssim 
			\inf_{v_h\in U_h} \|u-v_h\|_U
			+ \rho_X\|g\|_X 
			+\bigl(\rho_\Sigma+\rho_U+\rho_{\mathrm{nc}}\bigr)\left( \|f\|_U+\|g\|_{\Sigma_h'} \right).
		\end{aligned}
		$$
	\end{theorem}
	We remark that the discrete norm of $g$ itself on the right-hand side of the 
	error estimate may diverge as $h\to 0$.
	\begin{proof}[Proof of Theorem~\ref{thm:error}]
	    The $U$-orthogonal projection $Q_h: U \to U_h$ satisfies
		$$
		  \|u-Q_hu\|_U = \inf_{v_h\in U_h} \|u-v_h\|_U
		\quad\text{and}\quad
		  \|u-u_h\|_U^2	= \|u-Q_hu\|_U^2 +\|Q_h u - u_h\|_U^2.
		$$
		Thus, setting $e_h=Q_hu-u_h\in U_h$,
		it remains to estimate $\| e_h\|_U$.
		Let $(\phi, w) \in \Sigma^r \times U^r$ be the regular adjoint pair determined by $T(\phi,w)=-(0,e_h)$.  By \eqref{e:regularity} in Assumption~\ref{ass:regularity},
		the solution satisfies
		\begin{equation}
			\label{eq:dual-reg}
			\|\phi\|_{\Sigma^r}+\|w\|_{U^r} \lesssim\|e_h\|_U.
		\end{equation}
		Taking the second equation in $T(\phi,w)=-(0,e_h)$ with $e_h$ gives
		\begin{equation}
			\label{eq:error-step1}
			\|e_h\|_U^2= (e_h, Q_hu -u_h)_U =  (e_h, u-u_h)_U =(u,e_h)_U-b(\phi,u_h).
		\end{equation}
		The definition of the transposition solution in \eqref{def:transposed} yields
		$$
		    (u,e_h)_U=\langle g,\phi\rangle-(f,w)_U.
		$$ 
		The Fortin relation in Assumption~\ref{condition2}(i) and the first discrete
		equation in \eqref{eq:discrete} imply
		$$
		b(\phi,u_h)=b_h(\Pi_h\phi,u_h)=\langle g,\Pi_h\phi\rangle-a(\sigma_h,\Pi_h\phi).
		$$
		Substituting these two equalities into \eqref{eq:error-step1}
		gives
        \begin{equation}
			\label{eq:error-step4}
			\begin{aligned}
				\|e_h\|_U^2 & =
				\langle g,\phi-\Pi_h\phi\rangle
				+a(\sigma_h,\Pi_h\phi)-(f,w)_U.
			\end{aligned}
		\end{equation}
		Testing the second discrete equation in \eqref{eq:discrete} with $P_h w$ gives 
		$$b_h(\sigma_h,P_hw)+(f,P_hw)_U=0.$$
		Adding this and
		$0=R_h(\sigma_h; \phi, w) - a(\sigma_h,\phi)-b_h(\sigma_h, w)$
		to \eqref{eq:error-step4} yields
		$$
		\begin{aligned}
			\|e_h\|_U^2 &=	\langle g,\phi-\Pi_h\phi\rangle +a(\sigma_h,\Pi_h\phi-\phi) -(f,w-P_hw)_U \notag \\
			& \quad +R_h(\sigma_h; \phi, w) -b_h(\sigma_h,w-P_hw).
		\end{aligned}
		$$
		Using the extension of $a(\cdot, \cdot)$ to $H\times H$
		and $\Sigma,\Sigma_h \hookrightarrow H$, 
		we infer from Assumption~\ref{condition2}(ii)
		\begin{align*}
			|a(\sigma_h,\Pi_h\phi-\phi)| \lesssim  \|\sigma_h\|_{H} \| \Pi_h\phi-\phi\|_H
			\lesssim \rho_\Sigma \|\sigma_h\|_{H} \|\phi\|_{\Sigma^r} \lesssim \rho_\Sigma \| \sigma_h \|_{\Sigma,h} \|\phi\|_{\Sigma^r} .
		\end{align*}
		Continuity of $b_h$ and Assumption \ref{condition1} show
		$$
		|b_h(\sigma_h,w-P_hw)| \le \|\sigma_h\|_{\Sigma,h} \|w-P_hw\|_U
		\lesssim \rho_U \|\sigma_h\|_{\Sigma,h} \|w\|_{U^r}.
		$$
		This, combined with Assumptions~\ref{ass:regularity}
		and \ref{condition1}--\ref{condition3}, gives
		$$ \|e_h\|_U^2 \lesssim	\Bigl( \rho_U\|f\|_U +\rho_X\|g\|_X +( \rho_\Sigma +\rho_U+\rho_{\mathrm{nc}})\|\sigma_h\|_{\Sigma,h} \Bigr)\|e_h\|_U.$$
		The discrete stability \eqref{eq:discrete-stability} completes the proof.
	\end{proof}

	\begin{remark}
    Suppose that Assumption~\ref{condition2} holds more generally
    for all $z_h \in U$ rather than merely on the discrete subspace $U_h$. 
    Then, setting $e_h := u - u_h$ and choosing $(\phi, w)$ as the regular adjoint pair 
    satisfying $T(\phi, w) = -(0, e_h)$ directly in the proof of Theorem~\ref{thm:error},
    the remaining arguments 
    yield the following more explicit error estimate
    $$
   	\|u-u_h\|_U  \lesssim
   	 \rho_X\|g\|_X 
   	 + \bigl(\rho_\Sigma+\rho_U+\rho_{\mathrm{nc}}\bigr)\left( \|f\|_U+\|g\|_{\Sigma_h'} \right).
    $$
 	Except for the application to the Maxwell system,
 	this stronger assumption is satisfied for all model problems considered in this paper.
		If, in addition, $P_h$ satisfies
		$b_h(\tau_h, w - P_h w) =0 $,
		then we obtain the following error estimate 
		$$
		\begin{aligned}
			\|u-u_h\|_U \lesssim \rho_U\|f\|_U + \rho_X\|g\|_X +
			\bigl(\rho_\Sigma+\rho_{\mathrm{nc}}\bigr)
			\left(\|f\|_U+\|g\|_{\Sigma_h'} \right).
		\end{aligned}$$
	\end{remark}
	
	\begin{remark}[conforming reduction]
		\label{remark:conforming}
		Suppose $\Sigma_h \subset \Sigma$ and $b_h$, $b$ coincide.
		Then the discrete method \eqref{eq:discrete} is conforming
		and the error estimate in Theorem~\ref{thm:error} 
		holds with $\rho_{\mathrm{nc}} =0$.		
	\end{remark}
	
	\begin{remark}
		For the transposition solution $u$, a stress-like variable
		can be recovered $\sigma^{\rm tr}\in(\Sigma^r)'$ through
		$$
		\langle\sigma^{\rm tr},\tau\rangle :=\langle g,\tau\rangle-b(\tau,u) \qquad\text{ for all }\tau\in\Sigma^r.
		$$
		If we additionally assume 
		that the three properties in Assumption~\ref{condition2} hold
		for arbitrary $\tau\in\Sigma^r$, 
		then this definition, the first discrete equation
		in \eqref{eq:discrete},
		and the Fortin relation in Assumption~\ref{condition2}(i) give the identity
		$$
		\langle\sigma^{\rm tr},\tau\rangle-a(\sigma_h,\tau) =
		\langle g,\tau-\Pi_h\tau\rangle +a(\sigma_h,\Pi_h\tau-\tau) -b(\tau,u-u_h).
		$$
		We then see
		$$
		|\langle\sigma^{\rm tr},\tau\rangle-a(\sigma_h,\tau)| \lesssim
		\Bigl(\|u-u_h\|_U +\rho_X\|g\|_X +\rho_\Sigma\|\sigma_h\|_{\Sigma,h} \Bigr)\|\tau\|_{\Sigma^r}.
		$$
		Taking the supremum over nonzero $\tau\in\Sigma^r$ and substituting the error estimate of $\|u-u_h\|_U$ in Theorem \ref{thm:error} proves that the error
		that
		 $$
				\|\sigma^{\mathrm{tr}}-a(\sigma_h,\cdot)\|_{(\Sigma^r)'}
		$$
		is controlled by the upper bound from Theorem~\ref{thm:error}.
	\end{remark}

	\begin{remark}
	 The discretization parameter $h$ can be localized in concrete applications
	 where it often describes a local mesh size of a finite element
	 triangulation. In that case, $h$ is an $L^\infty$ function that gives rise
	 to well-defined weighted norms $\|\rho_U \cdot\|_U$ and so forth.
	 For the sake of a simple presentation, we refrain from this formalization
	 step and leave the usually simple generalization to particular 
	 FEM applications.
	\end{remark}

	\section{Application to the Laplacian}
	\label{sec:Laplacian}
	
	Let $\Omega \subseteq \mathbb{R}^m$ with $m=\{2,3\}$
	be an open, bounded, connected, polytopal Lipschitz domain.
	The Poisson equation with Dirichlet boundary condition reads
	\begin{equation}
		\label{eq:laplace}
		-\Delta u = f \quad \text{in }\Omega, 
		\qquad u= g \quad  \text{ on }\partial\Omega.
	\end{equation}
	We illustrate a discretization with low-order Raviart--Thomas FEM,
	but the error estimate applies to many other mixed schemes as well.

	\subsection{Notation}
	\label{sec:notation}
We briefly fix some standard notation that will apply throughout the 
remaining parts of this article.
Given a finite-dimensional real vector space $Y$,
let $H^\ell(\dom; Y)$ denote the standard Hilbertian 
Sobolev space of real order 
$\ell \ge 0$ on the open domain $\dom$,
equipped with the norm $\Vert{}\cdot\Vert{}_{H^\ell(\dom)}$.
In particular, $L^2(\dom; Y)$ denotes the space of square-integrable functions associated with the inner product $(\cdot, \cdot)_{L^2(\dom)}$. 
The subspace $H_0^\ell(\dom; Y)$ denotes the closure of $C_0^\infty(\dom; Y)$ in $H^\ell(\dom; Y)$, and $H^{-\ell}(\dom; Y)$ denotes the dual space of $H_0^\ell(\dom; Y)$.
If $Y = \mathbb{R}$, we simply write $H^\ell(\dom)$, $L^2(\dom)$, and so forth.
The space of real $m\times m$ matrices is denoted by $\mathbb{M}$,
symmetric matrices are denoted by
$\mathbb{S} := \{M\in\mathbb M: M=M^\top\}$,
and the $m\times m$ unit matrix is denoted by $I$.
	For $0<s<1$
and  an $m$-dimensional Lipschitz domain $G$, define 
  $
    H^s(\partial G;Y)
    :=\{g\in L^2(\partial G;Y): \|g\|_{H^s(\partial G)}<\infty\}
   $,	
	where
	$$ |g|_{H^s(\partial  G)}^2 := \int_{\partial G}\int_{\partial G}
	\frac{|g(x)-g(y)|_Y^2} {|x-y|^{m-1+2s}} \,ds_x\,ds_y,
	\, \|g\|_{H^s(\partial G)}^2 := \|g\|_{L^2(\partial G)}^2+|g|_{H^s(\partial G)}^2. $$
It is assumed that the discrete spaces are defined over
simplicial triangulation $\mathcal{T}_h$ of $\bar{\Omega}$
from a shape-regular family, that they are quasi-uniform,
and generic constants may depend on their shape-regularity.
We note that the results of this paper remain true on more general
meshes that are only locally quasi-uniform because the critical
estimates can be localized.
However, for the sake of a simple exposition, we impose the
quasi-uniformity condition here and mention that the proofs
will generalize with minor modifications.
For any element $K \in \mathcal{T}_h$ and degree $\ell \ge 0$, $P_\ell(K; Y)$ denotes the space of $Y$-valued polynomials of degree at most $\ell$, with $P_\ell(\mathcal{T}_h; Y)$ denoting the associated
piecewise polynomial space with respect to $\mathcal{T}_h$.
Again, $Y$ is omitted in the notation when $Y = \mathbb{R}$.
The mesh size is denoted by $h := \max_{K \in \mathcal{T}_h} h_K$, where $h_K := \operatorname{diam}(K)$. Let $\mathcal{F}_h$, $\mathcal{E}_h$, and $\mathcal{V}_h$ denote the sets of hyperfaces, edges, and vertices in $\mathcal{T}_h$, respectively. Superscripts $\mathrm{i}$ and $\partial$ distinguish interior and boundary entities (e.g., $\mathcal{V}_h = \mathcal{V}_h^{i} \cup \mathcal{V}_h^\partial$). In addition, the local sets of faces, edges, and vertices of an element $K \in \mathcal{T}_h$ are denoted by $\mathcal{F}(K)$, $\mathcal{E}(K)$, and $\mathcal{V}(K)$, respectively.  The $L^2$-projection onto $P_\ell(\mathcal{T}_h)$ is denoted by $Q_h^\ell$. Similarly, $Q_F^\ell$ and $Q_e^\ell$ represent the $L^2$-projections onto $P_\ell(F)$ and $P_\ell(e)$ for a face $F$ and an edge $e$, respectively, and $Q_\partial^\ell$ denotes 
 the piecewise boundary $L^2$-projection.
	
	\subsection{Mixed method}
	
	The mixed formulation of \eqref{eq:laplace} is based on spaces 
	$$
	\Sigma = H(\operatorname{div},\Omega; \mathbb{R}^m) := \{ \tau \in L^2(\Omega; \mathbb{R}^m): \ddiv \tau \in L^2(\Omega) \} \text{ and } U=L^2(\Omega),
	$$
	with norms $\| \cdot\|_{\Sigma} = \| \cdot \|_{L^2(\Omega)}
	+ \lVert \operatorname{div} \cdot \rVert_{L^2(\Omega)}$ and $\| \cdot \|_U = \| \cdot \|_{L^2(\Omega)}$.
	The bilinear forms are
	$$ 
	a(\sigma,\tau) =(\sigma,\tau)_{L^2(\Omega)},
	\quad b(\tau,v) =(\ddiv \tau, v)_{L^2(\Omega)} \,
	         \text{ for all }  \sigma, \tau \in \Sigma, v \in U.
	$$
	With $H = L^2(\Omega; \mathbb{R}^m)$ normed by
	$\| \cdot \|_H = \| \cdot \|_{L^2(\Omega)}$,
	$a(\cdot, \cdot)$ extends to a bounded, symmetric, positive semi-definite form on $H \times H$.
	Assumption \ref{ass:saddle} follows from classical theory \cite{BoffiBrezziFortin2013}.
	Let $s \in (1/2,1]$ be an admissible elliptic shift exponent for the homogeneous
	Poisson Dirichlet problem,
	i.e., a number such that $f\in U$ in \eqref{eq:laplace} implies
	$\|u\|_{H^{1+s}(\Omega)}\lesssim \|f\|_U$.
	In two and three space dimensions, the existence of $s$ is proven
	in \cite{Grisvard1985,Dauge1988}.
	Let  $U^r = H^{1+s}(\Omega) \cap H_0^1(\Omega)$ with the norm $\| \cdot \|_{U^r} = \| \cdot \|_{H^{1+s}(\Omega)}$  and
	$$
	\Sigma^r = \Sigma \cap H^{s}(\Omega; \mathbb{R}^m)
	 \text{ with the norm }
	\|\cdot\|_{\Sigma^r} =
	     \|\cdot\|_{H^{s}(\Omega)}
	      + \lVert\operatorname{div}\cdot\rVert_{L^2(\Omega)}.
	$$  
	With this choice of spaces, 
	given any $z \in L^2(\Omega)$,  the solution to $T(\phi, w) = (0,z)$ is unique
	and satisfies \eqref{e:regularity},
	see, e.g., \cite{Grisvard1985,Dauge1988}
	or the summary result in \cite[Lemma 3.4]{Gao2025}. This regularity result proves Assumption \ref{ass:regularity}.

	Given $K \in \mathcal{T}_h$, we define the usual Raviart--Thomas space
	as 
	$
	RT_0(K) := \{ \tau: K\to\mathbb R^m: \tau(x) = \alpha + \beta x
	\text{ for }\alpha\in\mathbb R^m,\beta\in\mathbb R\}
	$
	and the
	stable conforming lowest-order Raviart--Thomas pair
	$$
	\Sigma_h = \{\tau\in\Sigma: \, \tau|_K\in RT_0(K) \,  \text{ for all } K\in\mathcal T_h\}
	 \quad\text{and}\quad
	  U_h = P_0(\mathcal{T}_h)
	  .
	$$
	Assumption \ref{ass:discrete-stability} holds for this discrete pair \cite{BoffiBrezziFortin2013}.
	Let $X=H^{t}(\partial\Omega)$ with $0\le {t}<1/2$ and $g \in X$. 
	In the discrete formulation, we have 
	$\langle g, \tau_h \rangle = (g, \tau_h \cdot n)_{L^2(\partial \Omega)}$
	for all $\tau_h \in \Sigma_h$, where $n$ denotes the outer unit normal vector of the boundary $\partial \Omega$.
		For $t=0$,  the Cauchy-Schwarz inequality and the inverse trace  inequality give
		$$ |\langle g,\tau_h \rangle| \le
		\|g\|_{L^2(\partial\Omega)} \|\tau_h\cdot n\|_{L^2(\partial\Omega)}
		\lesssim
		h^{-1/2} \|g\|_{L^2(\partial\Omega)} \|\tau_h\|_{L^2(\Omega)}.
		$$
		For $0 < t < 1/2$, there exists a lifting $G\in H^{1/2+t}(\Omega)$ such that  $G|_{\partial\Omega}=g$ and  $\|G\|_{H^{1/2+t}(\Omega)}\lesssim
		\|g\|_{H^{t}(\partial\Omega)}$. A generalized integration by parts gives 
		$$
		\langle g,\tau_h\rangle
		 = (\operatorname{div} \tau_h,G)_{L^2(\Omega)} + \langle \nabla G,  \tau_h \rangle_\Omega,
		$$
		where the last term is
		the duality product between $H^{t-1/2}(\Omega)$ and $H^{1/2-t}(\Omega)= H_0^{1/2-t}(\Omega)$. The Cauchy-Schwarz inequality then implies
		$$
		|(\operatorname{div} \tau_h,G)_{L^2(\Omega)} | 
		\leq \lVert \operatorname{div} \tau_h\rVert_{L^2(\Omega)} \| G\|_{L^2(\Omega)} 
		\lesssim \| \tau_h \|_{\Sigma} \| g\|_{H^t(\partial \Omega)}. 
		$$
		The inverse estimate in fractional Sobolev spaces shows
		$$
		 \langle \tau_h, \nabla G \rangle_\Omega \leq \| \nabla G\|_{H^{t-1/2}(\Omega)} \| \tau_h\|_{H^{1/2-t}(\Omega)} \lesssim h^{t-1/2}\| G\|_{H^{t+1/2}(\Omega)} \| \tau_h\|_{L^2(\Omega)}
		 .
		$$
        Thus, 
        $
        \langle g,\tau_h\rangle_{\partial\Omega} 
        \lesssim (1 + h^{t-1/2}) \| g\|_{H^t(\partial \Omega)} \| \tau_h \|_{\Sigma}
        $. 
		This, combined with the case $t=0$, shows
		the a~priori stability bound
		\begin{equation}
		\label{ds-po}
		\| \sigma_h \|_{\Sigma, h } + \| u_h \|_U 
			\lesssim \| f\|_{L^2(\Omega)} + h^{t-1/2} \| g\|_X 
			\text{ for all }  0 \leq t < 1/2.
		\end{equation}

	Let $P_h: U\to U_h$ be the $L^2$ projection onto piecewise constants.
	The Poincar\'e inequality implies
	$$
	\|w-P_hw\|_{L^2(\Omega)}
	\lesssim  h \|w\|_{U^r} \quad \text{ for all } w \in U^r,
	$$
	which verifies Assumption~\ref{condition1} with  $\rho_U=h$. 
	Let $\Pi_h:\Sigma^r\to\Sigma_h$ be the commuting Raviart--Thomas
	interpolation \cite{BoffiBrezziFortin2013}.
	Since $\operatorname{div}\Pi_h \phi  =P_h\operatorname{div}\phi$ for all $\phi \in \Sigma^r$, Assumption~\ref{condition2}(i) holds. Furthermore,
	$$
	\|\phi-\Pi_h\phi \|_{L^2(\Omega)} \lesssim
	h^{s}\|\phi\|_{\Sigma^r} \quad \text{ for all } \phi \in \Sigma^r.
	$$
	This gives Assumption~\ref{condition2}(ii) with $\rho_\Sigma=h^{s}$.
	Since the
	Raviart--Thomas projection $\Pi_h$ is defined such that
	$(\Pi_h \phi)\cdot n = Q_\partial^0(\phi \cdot n)$ holds
	for all $\phi \in \Sigma^r$ on $\partial\Omega$, we infer
	$$
	\begin{aligned}
		|\langle g,(\phi-\Pi_h\phi)\cdot n\rangle|
		&\le \|g- Q_\partial^0 g\|_{L^2(\partial\Omega)}
		\|(1-Q_\partial^0)(\phi\cdot n)\|_{L^2(\partial\Omega)} \\
		&\lesssim h^{s+t-1/2} \|g\|_{H^t(\partial\Omega)}\|\phi\|_{H^s(\Omega)}.
	\end{aligned}
	$$
	Thus $\rho_X=h^{s+t-1/2}$. Since $\Sigma_h \subset \Sigma$, Remark \ref{remark:conforming} applies and $\rho_{\rm nc} = 0$.
	We collect these findings as the following result, implied by
        Theorem~\ref{thm:error}.
	\begin{corollary}
		\label{co:laplace}
        Let $\Omega\subset\mathbb R^m$ be a bounded, open
        Lipschitz polytope and assume there exists some
        $1/2 < s \leq 1$ such that the homogeneous Dirichlet problem
        for the Laplacian satisfies elliptic $H^{1+s}(\Omega)$ regularity.
		Let $f\in L^2(\Omega)$ and $g\in H^t(\partial\Omega)$
        with $0\leq t < 1/2$.
        Then, the transposition solution $u$ of the Laplacian
	    $-\Delta u= f$ in $\Omega$ with boundary data $g$
		on $\partial\Omega$
		and the lowest-order Raviart-Thomas solution
        $(\sigma_h,u_h)$ to \eqref{eq:discrete}
		satisfy the error bound
		$$
		\|u-u_h\|_{L^2(\Omega)} \lesssim h^{s}\|f\|_{L^2(\Omega)} 
        + h^{s+t-1/2}\|g\|_{H^{t}(\partial\Omega)}
        .
		$$
	\end{corollary}

	\begin{remark}
		The $L^2$ error estimate of $u-u_h$
		in Corollary \ref{co:laplace} coincides with
		the result provided in \cite[Theorem 2.2]{Gao2025}
		and \cite[Corollary 5.1]{Gao2025}.
		The proof shown here can dispense with a regularization procedure.
	\end{remark}

	\begin{example}
		\label{ex:laplace}
    We introduce some notation that will be used in all numerical
    illustrations throughout this paper.
    We consider the planar domains
	$\Omega_S$ and $\Omega_L$ and their three-dimensional extrusions
	$\Omega_C$ and $\Omega_P$, defined by
	\begin{equation}
		\label{eq:computational-domains}
		\begin{aligned}
			\Omega_S &= (0,1)^2,
			&
			\Omega_L &= (-1/2,1/2)^2
			\setminus \bigl([0,1/2]\times[-1/2,0]\bigr),\\
			\Omega_C &= \Omega_S\times(0,1),
			&
			\Omega_P &= \Omega_L\times(0,1).
		\end{aligned}
	\end{equation}
	Thus, $\Omega_S$ is the unit square, $\Omega_L$ is an L-shaped
	domain, $\Omega_C$ is the unit cube, and $\Omega_P$ is an
	L-shaped prism.
	The distinguished corner is the origin in both planar domains;
	it is convex for $\Omega_S$ and reentrant for $\Omega_L$.
	The corresponding edge in the three-dimensional domains is
	$\{(0,0,z):0<z<1\}$.
	The opening angles at these corners, or along these edges, are
	$\omega_S=\omega_C=\pi/2$ and $\omega_L=\omega_P=3\pi/2$.
	We use polar coordinates $(r,\theta)$ in the $(x,y)$-plane,
	centered at the origin.
	The angular ranges are $0<\theta<\omega_S$ on $\Omega_S$
	and $0<\theta< \omega_{L}$ on $\Omega_L$.
	The same convention applies to the horizontal cross sections
	of $\Omega_C$ and $\Omega_P$, respectively.
	The initial meshes are shown in Figure~\ref{fig:initial},
	and subsequent meshes are obtained by uniform refinement.
	
			 \begin{figure}[tbp]
		\centering
		%------------------------------------------------
		% (c) 2D: Unit square
		%------------------------------------------------
		\begin{subfigure}{0.23\textwidth}
			\centering
			\begin{tikzpicture}[scale=1.5, line cap=round, line join=round]
				\draw (0,0) rectangle (1,1);
				\draw (0.5,0) -- (0.5,1);
				\draw (0,0.5) -- (1,0.5);
				\draw (0,0) -- (0.5,0.5);
				\draw (0.5,0) -- (1,0.5);
				\draw (0,0.5) -- (0.5,1);
				\draw (0.5,0.5) -- (1,1);
			\end{tikzpicture}
			\caption{Unit square.}
			\label{fig:unit-square-mesh}
		\end{subfigure}\hfill
		%------------------------------------------------
		% (d) 2D: L-shaped domain
		%------------------------------------------------
		\begin{subfigure}{0.25\textwidth}
			\centering
			\begin{tikzpicture}[scale=0.75, line cap=round, line join=round]
				\draw (-1,-1) -- (0,-1) -- (0,0) -- (1,0)
				-- (1,1) -- (-1,1) -- cycle;
				\draw (-1,0) -- (0,0);
				\draw (0,0) -- (0,1);
				\draw (-1,-1) -- (0,0);
				\draw (-1,0) -- (0,1);
				\draw (0,0) -- (1,1);
			\end{tikzpicture}
			\caption{L-shaped domain.}
			\label{fig:l-domain-mesh}
		\end{subfigure} \hfill
		%------------------------------------------------
		% (a) 3D: Unit cube
		%------------------------------------------------
		\begin{subfigure}{0.23\textwidth}
			\centering
			\begin{tikzpicture}[
				x={(0.65cm,0cm)},
				y={(0.28cm,0.17cm)},
				z={(0cm,0.65cm)},
				scale=0.9,
				line cap=round,
				line join=round
				]
				\tikzset{
					mesh/.style={line width=0.25pt},
					bd/.style={very thick}
				}
				
				% Surface mesh on the visible faces.
				\foreach \a in {0,1,2} {
					\draw[mesh] (\a,0,0) -- (\a,0,2);
					\draw[mesh] (0,0,\a) -- (2,0,\a);
					
					\draw[mesh] (2,\a,0) -- (2,\a,2);
					\draw[mesh] (2,0,\a) -- (2,2,\a);
					
					\draw[mesh] (\a,0,2) -- (\a,2,2);
					\draw[mesh] (0,\a,2) -- (2,\a,2);
				}
				
				% Triangulation on the front face.
				\draw[mesh] (0,0,0) -- (1,0,1);
				\draw[mesh] (1,0,0) -- (2,0,1);
				\draw[mesh] (0,0,1) -- (1,0,2);
				\draw[mesh] (1,0,1) -- (2,0,2);
				
				% Triangulation on the right face.
				\draw[mesh] (2,0,0) -- (2,1,1);
				\draw[mesh] (2,1,0) -- (2,2,1);
				\draw[mesh] (2,0,1) -- (2,1,2);
				\draw[mesh] (2,1,1) -- (2,2,2);
				
				% Triangulation on the top face.
				\draw[mesh] (0,0,2) -- (1,1,2);
				\draw[mesh] (1,0,2) -- (2,1,2);
				\draw[mesh] (0,1,2) -- (1,2,2);
				\draw[mesh] (1,1,2) -- (2,2,2);
			\end{tikzpicture}
			\caption{Unit cube.}
			\label{fig:maxwell-cube-mesh}
		\end{subfigure}\hfill
		%------------------------------------------------
		% (b) 3D: L-shaped prism
		%------------------------------------------------
		\begin{subfigure}{0.23\textwidth}
			\centering
			\begin{tikzpicture}[
				x={(0.65cm,0cm)},
				y={(0.28cm,0.17cm)},
				z={(0cm,0.65cm)},
				scale=0.9,
				line cap=round,
				line join=round
				]
				\tikzset{
					mesh/.style={line width=0.25pt},
					bd/.style={very thick}
				}
				
				% Front face y=0, 0 <= x <= 1.
				\foreach \a in {0,1} {
					\draw[mesh] (\a,0,0) -- (\a,0,2);
				}
				\foreach \a in {0,1,2} {
					\draw[mesh] (0,0,\a) -- (1,0,\a);
				}
				\draw[mesh] (0,0,0) -- (1,0,1);
				\draw[mesh] (0,0,1) -- (1,0,2);
				
				% Outer right face x=2, 1 <= y <= 2.
				\foreach \a in {1,2} {
					\draw[mesh] (2,\a,0) -- (2,\a,2);
				}
				\foreach \a in {0,1,2} {
					\draw[mesh] (2,1,\a) -- (2,2,\a);
				}
				\draw[mesh] (2,1,0) -- (2,2,1);
				\draw[mesh] (2,1,1) -- (2,2,2);
				
				% Reentrant face y=1, 1 <= x <= 2.
				\foreach \a in {1,2} {
					\draw[mesh] (\a,1,0) -- (\a,1,2);
				}
				\foreach \a in {0,1,2} {
					\draw[mesh] (1,1,\a) -- (2,1,\a);
				}
				\draw[mesh] (1,1,0) -- (2,1,1);
				\draw[mesh] (1,1,1) -- (2,1,2);
				
				% Reentrant face x=1, 0 <= y <= 1.
				\foreach \a in {0,1} {
					\draw[mesh] (1,\a,0) -- (1,\a,2);
				}
				\foreach \a in {0,1,2} {
					\draw[mesh] (1,0,\a) -- (1,1,\a);
				}
				\draw[mesh] (1,0,0) -- (1,1,1);
				\draw[mesh] (1,0,1) -- (1,1,2);
				
				% Top L-shaped face z=2.
				\draw[mesh]
				(0,0,2) -- (1,0,2) -- (1,1,2)
				-- (2,1,2) -- (2,2,2) -- (0,2,2) -- cycle;
				
				\draw[mesh] (1,0,2) -- (1,2,2);
				\draw[mesh] (0,1,2) -- (2,1,2);
				
				% Surface triangulation.
				\draw[mesh] (0,0,2) -- (1,1,2);
				\draw[mesh] (0,1,2) -- (1,2,2);
				\draw[mesh] (1,1,2) -- (2,2,2);
			\end{tikzpicture}
			\caption{L-shaped prism.}
			\label{fig:l-prism-mesh}
		\end{subfigure}
		\caption{Initial meshes in two and three dimensions.}
		\label{fig:initial}
	\end{figure}
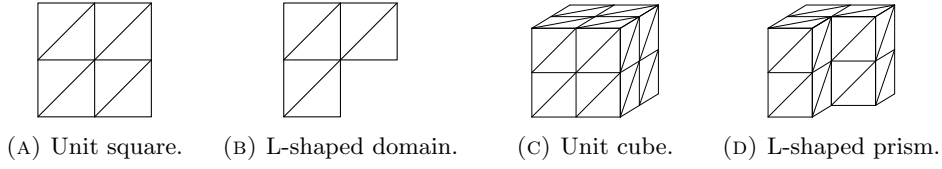
	For the actual example in the context of the Poisson equation,
	we consider the three-dimensional
	domains $\Omega_C$ and $\Omega_P$.
		We set $\alpha=-1/3$ and
		define the singular function $\psi$ and the exact solution by
		\begin{equation}
			\label{eq:poisson-3d-rough-exact}
			\psi(r,\theta)=r^\alpha\sin(\alpha\theta), \qquad
			u(x,y,z)=\psi(r,\theta)z(1-z).
		\end{equation}
		The right-hand side is
		$f(x,y,z)=2\psi(r,\theta)$. The induced boundary datum satisfies
		$$
		g=u|_{\partial\Omega} \in H^{1/6- \nu}(\partial\Omega) 
		\text{ for any } \nu >0.
		$$
		
		Given that the shift regularity indices for these domains are $s_C = 1$ and $s_P = 2/3$
		\cite{Dauge1988, Grisvard1985},
		the predicted $L^2$ convergence rates from Corollary~\ref{co:laplace} (with $t = 1/6 - \nu$)
		are (almost) $2/3$ and $1/3$, respectively. Table~\ref{tab:laplace} displays the $L^2$ errors
		along with experimental convergence orders,
		which conform to these predictions.
		The evaluation of integrals involving singular functions
		is done via composite Gaussian quadrature with
		local subdivision near the singular edge.
		\qed
		
		\begin{table}[tbp]
			\centering\small
			\begin{tabular}{ccccc}
				\toprule
				& \multicolumn{2}{c}{Unit cube}
				& \multicolumn{2}{c}{L-shaped prism}\\
				\cmidrule(lr){2-3}\cmidrule(lr){4-5}
				$h_\ell/\sqrt{3}$
				& $\|u-u_h\|_{L^2(\Omega)}$ & order
				& $\|u-u_h\|_{L^2(\Omega)}$ & order\\
				\midrule
				$2^{-1}$ & $2.2668\mathrm{e}{-2}$ & --
				& $6.0814\mathrm{e}{-2}$ & --\\
				$2^{-2}$ & $1.3396\mathrm{e}{-2}$ & $0.7589$
				& $3.7741\mathrm{e}{-2}$ & $0.6883$\\
				$2^{-3}$ & $7.8087\mathrm{e}{-3}$ & $0.7786$
				& $2.4190\mathrm{e}{-2}$ & $0.6417$\\
				$2^{-4}$ & $4.6202\mathrm{e}{-3}$ & $0.7571$
				& $1.6439\mathrm{e}{-2}$ & $0.5573$\\
				$2^{-5}$ & $2.7804\mathrm{e}{-3}$ & $0.7327$
				& $1.1767\mathrm{e}{-2}$ & $0.4824$\\
				$2^{-6}$ & $1.6969\mathrm{e}{-3}$ & $0.7124$
				& $8.7462\mathrm{e}{-3}$ & $0.4280$\\
				$2^{-7}$ & $1.0466\mathrm{e}{-3}$ & $0.6973$
				& $6.6644\mathrm{e}{-3}$ & $0.3922$\\
				$2^{-8}$ & $6.5020\mathrm{e}{-4}$ & $0.6867$
				& $5.1577\mathrm{e}{-3}$ & $0.3697$\\
				\bottomrule
			\end{tabular}
			\caption{$L^2$ errors and observed convergence orders for the
				$RT_0$-$P_0$ mixed approximation of the three-dimensional Poisson
				problem with boundary data
				$g \in H^{1/6-\nu}(\partial\Omega)$.}
			\label{tab:laplace}
		\end{table}
		
	\end{example}

	\subsection{The Crouzeix--Raviart method}
	\label{subsec:CR}
	The nonconforming $P_1$ method due to Crou\-zeix--Raviart
	offers an alternative
	scheme that allows for a local $L^2$ approximation of irregular
	boundary data. It is well known
	\cite{ArnoldBrezzi1986,Marini1985} that this discretization
	is in one-to-one correspondence with the mixed Raviart--Thomas
	method, and we can thereby derive an error bound for that scheme
	as a corollary of the previous result.
	We mention that, alternatively, a direct proof without the
	detour over the mixed method can be worked out as well.
	
		For $g\in X$, let 
	$V_{h,g}^{\rm nc}$ denote the Crouzeix--Raviart finite element
	space with boundary values prescribed by $g$.
	This affine space is spanned by the piecewise affine functions
	that satisfy $\int_F [v_h]ds =0$ for each interior face
	$F\in\mathcal F_h^i$ (here brackets denote the interface
	jump) and $\int_F (g-v_h)ds=0$ for each boundary
	face $F\in\mathcal F_h^\partial$.
	Then, $V_{h,0}^{\rm nc}$ denotes the corresponding homogeneous space. 
	The nonconforming method seeks
	$u_h^{\rm nc}\in V_{h,g}^{\rm nc}$ such that
	\begin{equation}\label{eq:cr-laplace}
		(\nabla_hu_h^{\rm nc},\nabla_hv_h)_{L^2(\Omega)}
		= (f,v_h)_{L^2(\Omega)} \qquad \text{for all }v_h\in V_{h,0}^{\rm nc},
	\end{equation}
	where $\nabla_h$ denotes the piecewise gradient with respect to $\mathcal{T}_h$.
    We introduce the 
	auxiliary discrete solution $\tilde u_h^{\rm nc}\in V_{h,g}^{\rm nc}$ satisfying
	\eqref{eq:cr-laplace} with $f$ replaced by $f_h = Q_h^0f$.

	Let $(\sigma_h,u_h)\in\Sigma_h\times U_h$ denote the
	$RT_0$--$P_0$ solution introduced in the previous subsection.
	The classical equivalence due to \cite{Marini1985} 
	states
	   \begin{equation}
			\label{eq:cr-rt}
			\sigma_h = \nabla_h\tilde u_h^{\rm nc} - m^{-1}f_hr_h,
			\qquad u_h = Q_h^0\tilde u_h^{\rm nc} + m^{-2}f_hQ_h^0(|r_h|^2).
		\end{equation}
	 where, for every $K\in\mathcal T_h$, 
	 $r_h|_K(x):=x-x_K$ with the barycenter $x_K$ of $K$.
	
	\begin{corollary}
		\label{cor:cr-poisson-rough}
		Let the conditions of Corollary~\ref{co:laplace} hold.
		Then the transposition solution $u$ of the Laplacian
		$-\Delta u= f$ in $\Omega$ with boundary data $g$
		on $\partial\Omega$ and the
		Crouzeix--Raviart solution to \eqref{eq:cr-laplace} satisfy the error bound
		$$
		\|u-u_h^{\rm nc}\|_{L^2(\Omega)} \lesssim h^s\|f\|_{L^2(\Omega)} + h^{s+t-1/2}\|g\|_{H^t(\partial\Omega)}.
		$$
	\end{corollary}
	
	\begin{proof}
		By the triangle inequality,
		\begin{equation}
			\|u-u_h^{\rm nc}\|_{L^2(\Omega)} \le \|u-u_h\|_{L^2(\Omega)} +
			\|u_h-\tilde u_h^{\rm nc}\|_{L^2(\Omega)} + \|\tilde u_h^{\rm nc}-u_h^{\rm nc}\|_{L^2(\Omega)}.
			\label{eq:cr-rt-triangle}
		\end{equation} 
	    A triangle inequality, \eqref{eq:cr-rt}, and the element-wise Poincar\'e inequality give
		$$
		\begin{aligned}
			\|u_h-\tilde u_h^{\rm nc}\|_{L^2(\Omega)} & \le \|u_h-  Q_h^0  \tilde u_h^{\rm nc}\|_{L^2(\Omega)} + \|  Q_h^0  \tilde u_h^{\rm nc} - \tilde u_h^{\rm nc}\|_{L^2(\Omega)} \\
			&\lesssim  h^2\|f_h\|_{L^2(\Omega)} +h\|\nabla_h\tilde u_h^{\rm nc}\|_{L^2(\Omega)}.
		\end{aligned}
		$$
		The combination of \eqref{eq:cr-rt} and the stability result 
		\eqref{ds-po} shows
		$$
		\|\nabla_h\tilde u_h^{\rm nc}\|_{L^2(\Omega)} \lesssim \|f_h\|_{L^2(\Omega)} + h^{t-1/2}\|g\|_{H^t(\partial\Omega)}.
		$$
		This implies 
		$$
		\|u_h-\tilde u_h^{\rm nc}\|_{L^2(\Omega)}  \lesssim  h\|f\|_{L^2(\Omega)}+h^{t+1/2}\|g\|_{H^t(\partial\Omega)}.
		$$
		Set $e_h:=\tilde u_h^{\rm nc}-u_h^{\rm nc}\in V_{h,0}^{\rm nc}$.
		Subtracting the two versions of \eqref{eq:cr-laplace}
		with $f$ and $f_h$ on the right-hand side, respectively,
		and testing with $e_h$ gives
		$
		\|\nabla_he_h\|_{L^2(\Omega)}^2=(f_h-f,e_h)_{L^2(\Omega)}=(f_h-f,e_h-Q_h^0e_h)_{L^2(\Omega)}.
		$
		Hence, the approximation property of $Q_h^0$ implies
		$\|\nabla_he_h\|_{L^2(\Omega)} \lesssim h\|f-f_h\|_{L^2(\Omega)}$.
		The discrete Poincar\'e inequality
		\cite[(10.6.14)]{BrennerScott2008}
		therefore yields
		$$
		\|\tilde u_h^{\rm nc}-u_h^{\rm nc}\|_{L^2(\Omega)}
		\lesssim \|\nabla_he_h\|_{L^2(\Omega)}
		\lesssim 
	   h\|f-f_h\|_{L^2(\Omega)}
		\lesssim
		h\|f\|_{L^2(\Omega)}.
		$$
		Substituting the preceding estimates and the estimate in Corollary \ref{co:laplace} into \eqref{eq:cr-rt-triangle} and using $s\le1$ completes the proof.
	\end{proof}

	\section{Application to the Stokes equations}
	\label{sec:Stokes}
	
	Let $\Omega$ be a domain as in Section~\ref{sec:Laplacian}.
	We consider the stationary Stokes problem
	\begin{equation}
		\label{eq:stokes}
		-\Delta u + \nabla p = f , \quad 
		\ddiv u =0 \text{ in } \Omega, \quad \text{ and } \quad 
		u=g  \text{ on }\partial \Omega,
	\end{equation}
	assuming that $g$ satisfies the compatibility condition $(g \cdot n,1)_{L^2(\partial \Omega)} =0$.
	Following \cite{CaiWang2010},
	we introduce the pseudo-stress $\sigma = -p I+ \nabla u$ as 
	an auxiliary variable, which is not necessarily symmetric. 
	The Stokes equation \eqref{eq:stokes} can be rewritten as 
	$$
	\Div \sigma = -f, \quad 
	\mathcal{A} \sigma - \nabla u = 0 \text{ in } \Omega, \quad \text{ and } \quad 
	u=g  \text{ on }\partial \Omega,
	$$
	where $\Div$ acts row-wise and $\mathcal{A}$ 
	is the deviatoric tensor defined by 
	$\mathcal{A} \tau  = \tau - m^{-1} \operatorname{tr}\tau  I$
	for any matrix $\tau$. 
	The incompressibility constraint $\operatorname{div} u = \operatorname{tr}(\nabla u) = 0$ is enforced, and the pressure trace $\operatorname{tr}\sigma = -mp$ is uniquely determined by imposing $(\operatorname{tr}\sigma, 1)_{L^2(\Omega)} = 0$. The mixed formulation is based on the spaces
	$$
	\Sigma = \{   \tau \in H(\Div, \Omega; \mathbb{M}): (\operatorname{tr} \tau, 1 )_{L^2(\Omega)} = 0 \}, \quad U= L^2(\Omega; \mathbb{R}^m)
	$$
	with norms $ \|\cdot\|_{\Sigma} = \|\cdot\|_{L^2(\Omega)} + \|\operatorname{Div}\cdot\|_{L^2(\Omega)}$ and  $\|\cdot\|_{U} = \| \cdot \|_{L^2(\Omega)}$.
	The bilinear forms are 
	$$
	a(\sigma, \tau)  = (\mathcal{A} \sigma, \tau)_{L^2(\Omega)}, \quad 
	b(\tau,v) = (\Div \tau, v)_{L^2(\Omega)}  \text{ for all } \sigma, \tau \in \Sigma, v \in U.
	$$
	We take $H = L^2(\Omega; \mathbb{M})$.
	Assumption~\ref{ass:saddle} follows from the pseudostress formulation
	of the Stokes problem \cite{CaiWang2010}.
	Let $s  \in (1/2,1]$ be an admissible elliptic 
	shift exponent for the homogeneous Dirichlet problem of the Stokes equation.
	Its existence is shown in \cite{Grisvard1985,Dauge1989},
	see also
	\cite{Girault1979} and \cite[Section~2.2]{Apel2026}.
	Let  $U^r = H^{1+s}(\Omega; \mathbb{R}^m) \cap H_0^1(\Omega; \mathbb{R}^m)$ with the norm $\| \cdot \|_{U^r} = \| \cdot \|_{H^{1+s}(\Omega)}$  and 
	$$
	\Sigma^r = \Sigma \cap H^{s}(\Omega; \mathbb{M})
	   \text{ with the norm } \|\cdot\|_{\Sigma^r} = \|\cdot\|_{H^{s}(\Omega)} 
	        + \lVert\operatorname{Div}\cdot\rVert_{L^2(\Omega)}.
	$$ 
	Under these assumptions,
	given $z \in L^2(\Omega; \mathbb{R}^m)$, the solution to $T(\phi,w)=(0,z)$ 
	is unique and 
	satisfies \eqref{e:regularity}.
	
	Let $\mathcal T_h$ be a simplicial mesh as in the last section.
	Let $\widehat\Sigma_h$ denote the space of matrix-valued
	functions whose rows belong to the divergence-conforming Raviart--Thomas
	space described in Section~\ref{sec:Laplacian}.
	We consider the conforming discretizations  as
	$$
	\Sigma_h= \{ \tau_h\in\widehat\Sigma_h: (\operatorname{tr}\tau_h,1)_{L^2(\Omega)}=0 \} 
	\quad\text{and}\quad U_h=P_0(\mathcal T_h; \mathbb{R}^m)
	.
	$$
	The stability of this pairing is shown in \cite{CaiWang2010}.
	Let $X= \{  g \in H^t(\partial \Omega; \mathbb{R}^m): (g \cdot n, 1)_{L^2(\partial \Omega)} = 0  \}$ with $0\le t<1/2$. For $g\in X$ we define $\langle g,\tau\rangle
	= ( g,\tau n)_{L^2(\partial \Omega)}$.
	By the boundary-trace estimate established in the
    context of \eqref{ds-po},
	applied row-wise,
	$$
	\|\sigma_h\|_{\Sigma} + \|u_h\|_{L^2(\Omega)}
	\lesssim \|f\|_{L^2(\Omega)} + h^{t-1/2}\|g\|_{H^t(\partial\Omega)}.
	$$
	Let $P_h:U\to U_h$ denote the $L^2$-orthogonal projection,
	which verifies Assumption~\ref{condition1} with $\rho_U=h$.
	Let $\widehat\Pi_h$ be the row-wise commuting Raviart--Thomas projection. 
	To preserve the normalization of $\Sigma_h$, define
	$$
	\Pi_h\phi = \widehat\Pi_h\phi 
	 - (m\operatorname{meas}(\Omega))^{-1} (\operatorname{tr}\widehat\Pi_h\phi,1)_{L^2(\Omega)}  I
	 \quad  \text{ for all } \phi \in \Sigma^r.
	$$
	Since the correction is a constant multiple of the identity matrix,
	it has zero row-wise
	divergence. Therefore 
	$\operatorname{Div}\Pi_h\phi = P_h\operatorname{Div}\phi$, 
	and hence Assumption~\ref{condition2}(i) holds. 
	Furthermore, as in Section~\ref{sec:Laplacian},
	 Assumption~\ref{condition2}(ii) holds with
    $\rho_\Sigma=h^s$.
	The compatibility condition $( g \cdot n, 1)_{L^2(\partial \Omega)} = 0$ gives $\langle g, (\phi - \Pi_h \phi)n \rangle= \langle g, (\phi - \widehat \Pi_h \phi)n \rangle$. Then, following similar procedures for deriving Assumption~\ref{condition2}(iii) in the application of the Laplace equation, one can prove the following estimate
	$$	|\langle g,(\phi-\Pi_h\phi)n\rangle|
	\lesssim h^{s+t-1/2} \|g\|_{H^t(\partial\Omega)} \|\phi\|_{\Sigma^r}.$$
	Therefore $ \rho_X=h^{s+t-1/2}$.  Since $\Sigma_h \subset \Sigma$, Remark \ref{remark:conforming} applies and $\rho_{\rm nc} = 0$. 
	
	\begin{corollary}
		\label{cor:stokes}
		Let $\Omega\subset\mathbb R^m$ be a bounded, open Lipschitz polytope. Assume that, for some
		$1/2<s\leq1$, the homogeneous Dirichlet Stokes problem
		with $L^2$ right-hand sides satisfies
		$H^{1+s}(\Omega;\mathbb R^m)$ regularity for the velocity
		and $H^s(\Omega)$ regularity for the pressure normalized
		to have zero mean. Let $f\in L^2(\Omega;\mathbb R^m)$ and
		$g\in H^t(\partial\Omega;\mathbb R^m)$ with
		$0\leq t<1/2$, and assume $\int_{\partial\Omega}g\cdot n\,ds=0$.
		Let $u$ be the transposition velocity solution of the Stokes equation
                 \eqref{eq:stokes} and let $(\sigma_h,u_h)$ be the
		lowest-order Raviart--Thomas mixed approximation
		in the pseudostress--velocity formulation.
		Then
		\[
		\|u-u_h\|_{L^2(\Omega)}
		\lesssim
		h^s\|f\|_{L^2(\Omega)}
		+
		h^{s+t-1/2}\|g\|_{H^t(\partial\Omega)}.
		\]
	\end{corollary}
	
	\begin{remark}
		For convex polygonal domains in two dimensions
                we have $s=1$. This Corollary gives 
		$ \|u-u_h\|_{L^2(\Omega)} \lesssim h\|f\|_{L^2(\Omega)} + h^{1/2}\|g\|_{L^2(\partial\Omega)}$
		for $g\in L^2(\partial\Omega; \mathbb{R}^m)$,
		which agrees with the optimal very weak Stokes estimates in \cite{Apel2026}.
		We note that our error bound can dispense with logarithmic factors,
		which in \cite{Duran2020} are a consequence of regularizing the boundary
		data.
	\end{remark}
	
	\begin{remark}
		\label{remark:cr-p0}
		Let $V_{h,g}^{\rm nc}, V_{h,0}^{\rm nc}$ be the vector-valued Crouzeix--Raviart
		space defined as in Section~\ref{subsec:CR}.
		The discrete pressure space is $U_h^0 = P_0(\mathcal{T}_h) \cap L_0^2(\Omega)$. 
		The nonconforming $CR$--$P_0$ mixed method for the Stokes equation reads: find $(u_h^{\rm nc},p_h^{\rm nc})
		\in V_{h,g}^{\rm nc}\times U_h^0$ such that
		$$
		\left\{   
		\begin{aligned}
			(\nabla_hu_h^{\rm nc},\nabla_hv_h)_{L^2(\Omega)}
			-(p_h^{\rm nc},\operatorname{div}_hv_h)_{L^2(\Omega)}
			&=(f,v_h)_{L^2(\Omega)}
			&& \text{ for all } v_h\in V_{h,0}^{\rm nc},\\
			(\operatorname{div}_hu_h^{\rm nc},q_h)_{L^2(\Omega)}
			&=0		&& \text{for all } q_h\in U_h^0.
		\end{aligned} \right.
		$$
		Following similar procedures as in the proof in Section~\ref{subsec:CR}, one can obtain 
		$$
		\|u-u_h^{\rm nc}\|_{L^2(\Omega)}
		\lesssim
		h^s\|f\|_{L^2(\Omega)}
		+
		h^{s+t-1/2}
		\|g\|_{H^t(\partial\Omega)}.
		$$
	\end{remark}

	\begin{example}
		\label{ex:stokes1}
		For the two-dimensional Stokes problem, we use the
			domains $\Omega_S$ and $\Omega_L$ and the associated mesh
			families introduced in Example~\ref{ex:laplace}.
			Set $\alpha = -1/3$ and $\omega = \omega_S$ or $\omega_L$.
		The exact velocity and pressure are defined by
		$$
		u(r,\theta)	=	r^\alpha(		\Phi_1(\theta),		\Phi_2(\theta) )^T,	\qquad	\tilde p(r,\theta)	=  r^{\alpha-1}\Phi_p(\theta),
		$$
		where the pressure is normalized by $
		p=\tilde p- \int_\Omega\tilde p\,dx/\operatorname{meas}(\Omega)$ and 
		\begin{align*}
			\Phi_1(\theta)={}&-\sin(\alpha\theta)\cos\omega	-\alpha\sin\theta
			\cos\bigl(\alpha(\omega-\theta)+\theta\bigr)\\
			&\quad +\alpha\sin(\omega-\theta)	\cos(\alpha\theta-\theta)
			+\sin\bigl(\alpha(\omega-\theta)\bigr),\\
			\Phi_2(\theta)	={}&	-\sin(\alpha\theta)\sin\omega	-\alpha\sin\theta
			\sin\bigl(\alpha(\omega-\theta)+\theta\bigr)
			-\alpha\sin(\omega-\theta)
			\sin(\alpha\theta-\theta),\\
			\Phi_p(\theta)
			={}&
			2\alpha
			\left[
			\sin\bigl((\alpha-1)\theta+\omega\bigr)
			+
			\sin\bigl((\alpha-1)\theta-\alpha\omega\bigr)
			\right].
		\end{align*}
		The functions defined above satisfy $-\Delta u+\nabla p=0$ and $\operatorname{div}u=0$ 
		in $\Omega$.
		Since $u$ behaves like $r^{-1/3}$ near the origin, its trace $g$
                belongs to
		$
		H^{1/6-\nu}
		\bigl(\partial\Omega;\mathbb R^2\bigr)
                $
		for all $\nu>0$.
		Table~\ref{tab:stokes1} displays the velocity $L^2$-errors and convergence rates. 
		On the unit square ($s_S=1$), the convergence rates 
		are close to the theoretical order $2/3$, while on the L-shaped domain 
		they are close to the predicted rate $s_L + t - 1/2 = 0.211\ldots$
		(with $s_L = 0.544\ldots$ and $t = 1/6-\nu$).
		\qed
		
\begin{table}[tbp]
	\centering\small
	\begin{tabular}{ccccc}
		\toprule
		& \multicolumn{2}{c}{Unit square}
		& \multicolumn{2}{c}{L-shaped domain}\\
		\cmidrule(lr){2-3}\cmidrule(lr){4-5}
		$h_\ell/\sqrt{2}$
		& $\|u-u_h\|_{L^2(\Omega)}$ & order
		& $\|u-u_h\|_{L^2(\Omega)}$ & order\\
		\midrule
		$2^{-1}$ & $2.3733\mathrm{e}{-1}$ & --
		& $1.9129\mathrm{e}{-1}$ & --\\
		$2^{-2}$ & $1.6508\mathrm{e}{-1}$ & $0.5238$
		& $1.3947\mathrm{e}{-1}$ & $0.4558$\\
		$2^{-3}$ & $1.0903\mathrm{e}{-1}$ & $0.5984$
		& $1.0525\mathrm{e}{-1}$ & $0.4061$\\
		$2^{-4}$ & $7.0175\mathrm{e}{-2}$ & $0.6357$
		& $8.1780\mathrm{e}{-2}$ & $0.3640$\\
		$2^{-5}$ & $4.4760\mathrm{e}{-2}$ & $0.6488$
		& $6.5123\mathrm{e}{-2}$ & $0.3286$\\
		$2^{-6}$ & $2.8414\mathrm{e}{-2}$ & $0.6556$
		& $5.3082\mathrm{e}{-2}$ & $0.2949$\\
		$2^{-7}$ & $1.7985\mathrm{e}{-2}$ & $0.6598$
		& $4.4137\mathrm{e}{-2}$ & $0.2662$\\
		$2^{-8}$ & $1.1364\mathrm{e}{-2}$ & $0.6624$
		& $3.7242\mathrm{e}{-2}$ & $0.2451$\\
		\bottomrule
	\end{tabular}
	\caption{$L^2$ errors and observed convergence orders for the
		pseudo-stress $RT_0$-$P_0$ approximation of the two-dimensional
		Stokes problem with boundary data
		$g \in H^{1/6-\nu}(\partial\Omega)$.}
	\label{tab:stokes1}
\end{table}

	\end{example}

	\section{Application to linear elasticity}
	\label{sec:Elasticity}
	
	Let $\Omega$ be a domain as in Section~\ref{sec:Laplacian}.
	We consider the linear elasticity problem
	$$
	-\Div \mathcal C\varepsilon (u) = f
	\quad\text{in }\Omega
	\qquad\text{and}\qquad
	u=g \quad\text{on }\partial \Omega.
	$$
	Here $\varepsilon (u) = 1/2 (\nabla u + (\nabla u)^T)$ is the 
	symmetric part of the derivative matrix,
	and the elasticity tensor
	$\mathcal C \tau = 2\mu \tau  
	+ \lambda \operatorname{tr}\tau I$
	for any symmetric matrix $\tau$ with material parameters $\mu,\lambda>0$.
	The Hellinger--Reissner formulation is based on the spaces
	$$
	\Sigma =H(\Div,\Omega;\mathbb S) := \{  \tau \in L^2(\Omega;\mathbb{S}): \Div \tau \in L^2(\Omega; \mathbb{R}^m) \}, \quad U = L^2(\Omega;\mathbb{R}^m)
	$$
	equipped with norms $\|\cdot\|_{\Sigma}$ and $\| \cdot \|_U$
	as in Section~\ref{sec:Stokes}.
	Define the bilinear forms
	$$
	a(\sigma,\tau) =(\mathcal{C}^{-1} \sigma, \tau)_{L^2(\Omega)}, \quad 
	b(\tau, v) = (\Div \tau,  v)_{L^2(\Omega)} \text{ for all } \sigma,  \tau \in  \Sigma, v \in U.
	$$ 
	We let $H = L^2(\Omega; \mathbb{S})$.
	Standard theory \cite{arnold2002mixed,BoffiBrezziFortin2013}
	implies the well-posedness of the saddle-point formulation \eqref{eq:classical-saddle}.
	Let $s\in(1/2,1]$ be an admissible regularity-shift exponent for the
	homogeneous Dirichlet elasticity problem.
	In two space dimensions, the existence of such $s$ is shown in
	\cite{Grisvard1985,Grisvard1992}.
	We define
	$ U^r= H^{1+s}(\Omega;\mathbb R^m) \cap H_0^1(\Omega;\mathbb R^m)$ with $\|\cdot \|_{U^r}=\|\cdot \|_{H^{1+s}(\Omega)}$, and
	$$
	\Sigma^r = \Sigma \cap H^s(\Omega;\mathbb S) \text{ with the norm } \|\cdot\|_{\Sigma^r}  = \|\cdot\|_{H^s(\Omega)} + \|\operatorname{Div}\cdot\|_{L^2(\Omega)}.
	$$
		In three dimensions, the analysis below applies under the assumption
		that a corresponding regularity result holds.
	Under this hypothesis on $s$,
	Assumption~\ref{ass:regularity} is satisfied.
	
	Symmetric conforming mixed finite elements for tensor-valued problems 
	typically require high-order polynomial spaces, and the regularity required
	for their full approximation order is not satisfied for irregular 
	boundary data.
	This motivates the use of low-order nonconforming mixed finite element methods in the present setting.
	As an example, we explicitly describe the two-dimensional 
	reduced Arnold--Winther element from \cite{ArnoldWinther2003}.
	For $K\in\mathcal T_h$, let 
	$RM(K)$ denote the space of rigid-body motions as
	functions over $K$, i.e., linear combinations of constant
	vector fields and $(-x_2,x_1)$.
	For every interior face $F$, fix a unit normal $n_F$, 
	and denote by $[\tau_hn_F]$ the corresponding jump. 
	Define the discrete spaces
	$$
	\Sigma_h = \left\{ \tau_h\in \Sigma(\mathcal{T}_h): \int_F [\tau_hn_F]\cdot q\,ds=0 \text{ for all } q\in P_1(F;\mathbb R^2), F\in\mathcal F_h^i \right\}
	$$
	with 
	$$\Sigma(K) := \left\{ \tau\in P_2(K;\mathbb S): n^T\tau n|_F\in P_1(F) \text{ for every }F\in\mathcal F(K),  \operatorname{Div}\tau\in RM(K)\right\}$$
	and $U_h = RM(\mathcal{T}_h)$, the piecewise rigid-body motions.
	A three-dimensional analogue is 
	given in \cite{Arnold2014}. The subsequent analysis is formulated 
	uniformly for $m=2,3$ and relies only on the structural properties 
	shared by these elements. Throughout this section, 
	$F$ denotes an $(m-1)$-dimensional mesh facet. 
	
Since $\Sigma_h\not\subset H(\Div,\Omega;\mathbb S)$, we use the piecewise 
action $\Div_h$ of the divergence and set
$$ b_h(\tau,v) = (\Div_h\tau,v)_{L^2(\Omega)}, \quad
\text{ for all } \tau\in \Sigma + \Sigma_h, v\in U.$$
The discrete graph norm is $\|\tau_h\|_{\Sigma,h}^2
:= \|\tau_h\|_{L^2(\Omega)}^2 +
\lVert \Div_h\tau_h\rVert_{L^2(\Omega)}^2$ for any $\tau_h \in \Sigma_h$. 
The stability of this pair is established in
\cite{ArnoldWinther2003,Arnold2014}.
Let $X=H^t(\partial\Omega;\mathbb R^m)$ with $0\leq t<1/2$.
For $g\in X$, define the discrete boundary functional by
$\langle g,\tau_h\rangle=
(g,\tau_hn)_{L^2(\partial\Omega)}$ for $\tau_h\in\Sigma_h$.
For $t=0$, $ |\langle g,\tau_h\rangle| \le
\|g\|_{L^2(\partial\Omega)} \|\tau_hn\|_{L^2(\partial\Omega)}\lesssim
h^{-1/2} \|g\|_{L^2(\partial\Omega)} \|\tau_h\|_{L^2(\Omega)}$. For $0<t<1/2$, there exists a boundary lifting $G\in H^{1/2+t}(\Omega;\mathbb R^m)$ such that  $G|_{\partial\Omega}=g$ and  $\|G\|_{H^{1/2+t}(\Omega)}\lesssim
\|g\|_{H^{t}(\partial\Omega)}$. An integration by parts gives 
\begin{equation}
	\label{ds-el:1}
	\langle g,\tau_h\rangle = (\operatorname{Div}_h\tau_h,G)_{L^2(\Omega)} + \langle \tau_h,\varepsilon(G)\rangle_\Omega- \sum_{F\in\mathcal F_h^i} ([\tau_hn_F],G)_{L^2(F)}.
\end{equation}
Here, $\langle \tau_h,\varepsilon(G)\rangle_\Omega$ is the duality
product between $H^{1/2-t}(\Omega; \mathbb{S})$ and $H^{t-1/2}(\Omega; \mathbb{S})$.
The first term can be estimated by the Cauchy-Schwarz inequality $$
(\operatorname{Div}_h\tau_h,G)_{L^2(\Omega)} 
 \leq \lVert \Div_h \tau_h \rVert_{L^2(\Omega)} \| G\|_{L^2(\Omega)} \leq \|g\|_{H^t(\partial \Omega)} \| \tau_h \|_{\Sigma,h}.
$$
For the second term, the inverse estimate shows 
$$
\langle \tau_h,\varepsilon(G)\rangle_\Omega
\leq \| \tau_h \|_{H^{1/2-t}(\Omega)} \| \varepsilon (G) \|_{H^{t-1/2}(\Omega)} \lesssim h^{t-1/2}\|g\|_{H^{t}(\partial \Omega)}\| \tau_h\|_{L^2(\Omega)}. 
$$
By the discrete continuity conditions on $\Sigma_h$, it holds that
$$
([\tau_hn_F],G)_{L^2(F)} = ([\tau_hn_F],G - Q_F^1 G)_{L^2(F)} \leq \| [\tau_hn_F] \|_{L^2(F)} \| G - Q_F^1 G\|_{L^2(F)}.
$$
The trace, inverse, and interpolation error estimates
prove 
$([\tau_hn_F],G)_{L^2(F)} \lesssim h^{t-1/2} \| g\|_{H^{t}(\partial \Omega)} \| \tau_h\|_{\Sigma,h}$.
Substituting these estimates into \eqref{ds-el:1} 
and combining with the case $t=0$ shows
$$
\| \sigma_h \|_{\Sigma, h } + \| u_h \|_U \lesssim \| f\|_{L^2(\Omega)} + h^{t-1/2} \| g\|_X.
$$

Let $P_h:U\to U_h$ denote the
$L^2$-orthogonal projection, which satisfies Assumption~\ref{condition1}
with $\rho_U=h$.
We introduce the canonical interpolation operator $\Pi_h:\Sigma^r\to\Sigma_h$
defined locally by the degrees of freedom
\begin{equation}
	\label{eq:elasticity-interpolation}
	\int_F
	\bigl((\phi-\Pi_h\phi)n_F\bigr)\cdot q\,ds
	=0
	\qquad
	\text{for all }q\in P_1(F;\mathbb R^m), F \in \mathcal{F}_h.
\end{equation}
Then, $\operatorname{Div}_h\Pi_h\phi
=P_h\operatorname{Div}\phi$ for all $\phi \in\Sigma^r$, which implies Assumption~\ref{condition2}(i). The approximation estimate for the reduced interpolation operator follows from the scaling argument in \cite[Section 5]{ArnoldWinther2003} and interpolation
$$
\|\phi-\Pi_h\phi\|_{L^2(\Omega)}
\lesssim
h^s\|\phi\|_{\Sigma^r},
$$
and hence Assumption~\ref{condition2}(ii) holds with $\rho_\Sigma=h^s$.
By the definition of $\Pi_h$ in \eqref{eq:elasticity-interpolation}, it follows
$$
\langle g,\phi-\Pi_h\phi\rangle= \sum_{F \in \mathcal F_h^\partial} \int_F (g-Q_F^1g)\cdot(\phi-\Pi_h\phi)n\,ds \lesssim
h^{s+t-1/2}\|g\|_{H^t(\partial\Omega)} \|\phi\|_{\Sigma^r}.
$$
Thus Assumption~\ref{condition2}(iii) holds with $\rho_X=h^{s+t-1/2}$. It remains to verify the nonconforming consistency condition. 
 Let $(\phi,w)\in\Sigma^r\times U^r$ solve $a(\phi, \tau)+b(\tau,w)= 0$
 for all $\tau\in\Sigma$.
 Since $\mathcal{C}^{-1}\phi=\varepsilon(w)$, element-wise integration by parts, $w=0$ on $\partial\Omega$, and the moment continuity of $\tau_hn_F$ give for $R_h$ defined in Assumption~\ref{condition3} that
$$
	R_h(\tau_h;\phi,w)
	 = \sum_{F \in\mathcal F_h^i} \int_F [\tau_hn_F]\cdot w\,ds 
	=  \sum_{F\in\mathcal F_h^i} \int_F [\tau_hn_F]\cdot(w-Q_F^1w)\,ds.
$$
The trace approximation estimate and the polynomial inverse trace inequality give
$$
|R_h(\tau_h;\phi,w)| \lesssim h^s \|\tau_h\|_{\Sigma,h} \|w\|_{U^r} .
$$
Thus Assumption~\ref{condition3} holds with $\rho_{\rm nc}=h^s$.

	\begin{corollary}
		\label{cor:ncaw-transposition}
 Let $\Omega\subset\mathbb R^m$ be an open and bounded Lipschitz polytope and
 assume that, for some $1/2<s\le1$, the homogeneous Dirichlet elasticity problem
 with $L^2$ right-hand sides satisfies an
	$H^{1+s}(\Omega;\mathbb R^m)$ regularity assumption.
	Let $f\in L^2(\Omega;\mathbb R^m)$ and  $g\in H^t(\partial\Omega;\mathbb R^m)$ with  $0\le t< 1/2$.
	Then the transposition solution $u$ of the linear elasticity problem and 
	the reduced Arnold--Winther solution $(\sigma_h,u_h)$
	satisfy
		\begin{equation*}
			\|u-u_h\|_{L^2(\Omega)} \lesssim h^s\|f\|_{L^2(\Omega)} + h^{s+t-1/2} \|g\|_{H^t(\partial\Omega)}.
		\end{equation*}
	\end{corollary}
	
	\begin{example}
		We consider the two-dimensional linear elasticity problem on the
		domains $\Omega_S$ and $\Omega_L$ from
                Example~\ref{ex:laplace}.
		The Lam\'e parameters are fixed as $\lambda=10^6$ and $\mu=1$. 	The exact displacement field is constructed from a harmonic-gradient potential:
		$$ \phi(r,\theta)
		=r^{1+\alpha}\sin\bigl((1+\alpha)\theta + \pi/4 \bigr),
		\qquad u=\nabla\phi, \qquad \alpha=-1/3. $$
		The body force vanishes ($f = 0$), and the Dirichlet condition $g = u|_{\partial\Omega}$ is prescribed on the entire boundary. 
		It satisfies $g\in H^{1/6-\nu}(\partial\Omega; \mathbb{R}^2)$ 
		for any $\nu > 0$.
		On the unit square $\Omega_S$, the homogeneous Dirichlet dual problem
		admits the regularity shift $s_S=1$.
		It is known  \cite{Sandig1989}
		that on the L-shaped domain $\Omega_L$ with its reentrant angle 
		$\omega_L=3\pi/2$, the leading corner exponent $\kappa_1$ is the
		smallest positive root of
		$$
		(\lambda+\mu)^{-2}(\lambda+3\mu)^2
		\sin^2(\kappa\omega_L)
		-
		\kappa^2\sin^2\omega_L
		=0.
		$$
		For the chosen Lam\'e parameters, $\kappa_1=0.544\ldots$, 
		and every $s_L$ with $1/2<s_L<\kappa_1$ is an admissible
		regularity-shift exponent.
		Consequently, Corollary~\ref{cor:ncaw-transposition} yields the convergence orders $h^{2/3-\nu}$ and $h^{\kappa_1-1/3-\nu}$ 
               for any $\nu>0$ for $\Omega_S$ and $\Omega_L$, respectively, 
		where $\kappa_1-1/3=0.211\ldots$. 
		Table~\ref{tab:elasticity} displays the displacement $L^2$-errors
		and observed convergence orders for the reduced nonconforming
		Arnold--Winther method.
		On $\Omega_S$, the observed orders approach the reference value
		$2/3$.
		On $\Omega_L$, the orders on the finer meshes decrease towards
		the reference value $0.211\ldots$.
		These results are consistent with the reduced convergence rate
		associated with the reentrant corner.
		\qed
		
				\begin{table}[tbp]
			\centering\small
			\begin{tabular}{ccccc}
				\toprule
				& \multicolumn{2}{c}{Unit square} & \multicolumn{2}{c}{L-shaped domain}\\
				\cmidrule(lr){2-3}\cmidrule(lr){4-5}
				$h_\ell/\sqrt{2}$ & $\|u-u_h\|_{L^2(\Omega)}$ & order & $\|u-u_h\|_{L^2(\Omega)}$ & order\\
				\midrule
				$2^{-1}$ & $1.2419\mathrm{e}{-1}$ & -- & $2.0601\mathrm{e}{-1}$ & --\\
				$2^{-2}$ & $8.0908\mathrm{e}{-2}$ & $0.6182$ & $1.6149\mathrm{e}{-1}$ & $0.3512$\\
				$2^{-3}$ & $5.1950\mathrm{e}{-2}$ & $0.6391$ & $1.1902\mathrm{e}{-1}$ & $0.4403$\\
				$2^{-4}$ & $3.3111\mathrm{e}{-2}$ & $0.6498$ & $8.8904\mathrm{e}{-2}$ & $0.4209$\\
				$2^{-5}$ & $2.1010\mathrm{e}{-2}$ & $0.6562$ & $6.8541\mathrm{e}{-2}$ & $0.3753$\\
				$2^{-6}$ & $1.3295\mathrm{e}{-2}$ & $0.6602$ & $5.4658\mathrm{e}{-2}$ & $0.3265$\\
				$2^{-7}$ & $8.3989\mathrm{e}{-3}$ & $0.6626$ & $4.4836\mathrm{e}{-2}$ & $0.2858$\\
				$2^{-8}$ & $5.3003\mathrm{e}{-3}$ & $0.6641$ & $3.7529\mathrm{e}{-2}$ & $0.2566$\\
				$2^{-9}$ & $3.3427\mathrm{e}{-3}$ & $0.6651$ & $3.1826\mathrm{e}{-2}$ & $0.2378$\\
				\bottomrule
			\end{tabular}
			\caption{$L^2$ displacement errors and observed convergence orders for the reduced nonconforming Arnold--Winther approximation of the two-dimensional linear elasticity problem with $\lambda=10^6$, $\mu=1$ and boundary data $g \in H^{1/6-\nu}(\partial\Omega)$.}
			\label{tab:elasticity}
		\end{table}
	\end{example}

	\section{Application to the biharmonic equation}
	\label{sec:Biharmonic}
	
	Let $\Omega \subseteq \mathbb{R}^2$ be an open, bounded, simply-connected, 
	Lipschitz polygon. In this section, we only consider the planar case.
	We consider the biharmonic equation with clamped boundary conditions
	\begin{equation}
	\label{eq:biharmonic}
	\Delta^2 u = f \quad\text{in }\Omega \quad u= g_{1}, \quad \partial_{n} u = g_{2} \text{ on }\partial\Omega,
	\end{equation}
	where $\partial_n$ is the normal derivative along $\partial \Omega$ and  $\Delta^2$ is the biharmonic operator. 
	
	The mixed formulation is based on the moment variable $\sigma =- \nabla^2 u$ 
	and the spaces 
	$$
	\Sigma = H(\dDiv, \Omega; \mathbb S) := \{  \tau \in L^2(\Omega; \mathbb{S}): \dDiv  \tau \in L^2(\Omega)\} \text{ and } U = L^2(\Omega)
	$$ 
	with norms $\|\cdot\|_\Sigma = \| \cdot\|_{L^2(\Omega)} + \| \dDiv \cdot\|_{L^2(\Omega)}$ and $\| \cdot \|_U = \| \cdot\|_{L^2(\Omega)}$, respectively. 
	Define the bilinear forms  
	$$
	a(\sigma,\tau)  =(\sigma, \tau)_{L^2(\Omega)}, \quad 
	b(\tau, v) = (\dDiv \tau, v)  \text{ for all }  \sigma, \tau \in \Sigma, v \in U.
	$$ 
	We let $H = L^2(\Omega; \mathbb{S})$ with the $L^2$ norm.
	The well-posedness of this mixed formulation follows from \cite{ChenHuang2020}.
	It is known \cite{blum1980boundary,Grisvard1985} that
	the clamped biharmonic problem over a polygon 
	with homogeneous boundary data
	possesses $H^{2+s}(\Omega)$ regularity with some
	$1/2<s\leq 2$, 	if $f\in L^2(\Omega)$.
	We refer to $s$ as an admissible regularity index for the
	homogeneous clamped biharmonic problem and define
	$ U^r= H^{2+s}(\Omega) \cap H_0^2(\Omega)$ 
	with $\|\cdot \|_{U^r}=\|\cdot \|_{H^{2+s}(\Omega)}$,
	and
	$$
	\Sigma^r = \Sigma  \cap H^s(\Omega;\mathbb S) \text{ with the norm }
	\|\cdot\|_{\Sigma^r} =\|\cdot\|_{H^s(\Omega)} + 
	\lVert \dDiv \cdot\rVert_{L^2(\Omega)}.
	$$
	Then, for every $z\in L^2(\Omega)$,  the solution to $T(\phi, w)=(0,z)$ is unique and  satisfies \eqref{e:regularity},
	which implies Assumption \ref{ass:regularity}. 
	
	In this section, the regularity of traces will be described by
	the symbols $t_1$, $t_2$. Therefore, we can
	denote the unit tangent vector to $\partial\Omega$
	by $t=(-n_2,n_1)^T$ without conflicting notation.
	In this section, we will use the space
	$H^1(\partial\Omega)$ with norm
	$\|v\|_{H^1(\partial\Omega)}^2
	  = \|v\|_{L^2(\partial\Omega)}^2
	   + \| \partial_t v\|_{L^2(\partial\Omega)}^2$
	and
	$H^\nu(\partial\Omega)$ for $1<\nu<3/2$,
	with norm
	$\|v\|_{H^\nu(\partial\Omega)}^2
	  = \|v\|_{H^1(\partial\Omega)}^2
	   + | \partial_t v|_{H^{\nu-1}(\partial\Omega)}^2$
	for the boundary of the planar Lipschitz polygon $\Omega$.
	It is proven in \cite[Theorem 6.1]{ArnoldScottVogelius1988}
	that for these $\nu$ the space coincides with
	the trace space of $H^{1/2+\nu}(\Omega)$.
	For a scalar function $a$ and vector function $q$, define
	the following rotated-gradient (vector curl)
	operations in the plane
		$$
		\nabla^\perp a = (\partial_2 a, -\partial_1 a)^T, 
		\quad 
		\nabla^\perp q = \begin{pmatrix} \partial_2 q_1 & -\partial_1 q_1 \\ \partial_2 q_2 & -\partial_1 q_2 \end{pmatrix}, 
		\quad \nabla^\perp_{\mathbb S} q = \operatorname{sym}(\nabla^\perp q).
		$$
	
	In the following, we consider the conforming triangle element proposed by
        F\"uhrer and Heuer
	\cite{FuhrerHeuer2025}.
	Three-dimensional conforming elements can be found in 
        \cite{ChenHuang2022, HuMaZhang2021}, but are not discussed here.
	Let $\mathcal T_h$ be a triangulation of $\Omega$ as in prior
	sections.
	Following \cite{FuhrerHeuer2025},
	let $U_h = P_1(\mathcal{T}_h)$ and 
	$$
	\Sigma_h=\{\tau_h\in H(\dDiv,\Omega;\mathbb S):
	\tau_h|_K\in \nabla^\perp_{\mathbb S}(RT_2(K))\oplus xx^TP_1(K) 
	\text{ for all }K\in\mathcal T_h\},
	$$
	$RT_2(K)=P_2(K;\mathbb R^2)+xP_2(K)$ denotes the local Raviart--Thomas
	space of second order \cite{BoffiBrezziFortin2013}.
	The dimension of the local shape function space for $\Sigma_h$ is $15$.
	The pair $(\Sigma_h,U_h)$ is uniformly inf-sup stable
	\cite{FuhrerHeuer2025} and, thus, Assumption~\ref{ass:discrete-stability}
	is satisfied. Let $X:=H^{t_1}(\partial\Omega) \times H^{t_2}(\partial\Omega)$ with $1\leq t_1 < 3/2$ and $0 \leq t_2 < 1/2$.  Due to the low regularity of $\phi \in \Sigma^r$, the definition of the 
	trace pairing and the construction of the interpolation operator are 
	not obvious. For this, we introduce the following lemma.
			
		We note that the kernel of the symmetric curl equals
		the three-dimensional space $\mathit{RT}$ in the plane,
		spanned by the constant vector fields
		and the identity mapping $x$.
		In what follows we 
		denote the space of $H^1$ fields that are $L^2$-orthogonal to
		$\mathit{RT}$ by $H^1(\Omega;\mathbb R^2)/\mathit{RT}$.
		
		\begin{lemma}[regular decomposition]
			\label{lem:decomposition}
			Let $\Omega\subset\mathbb R^2$ be an open, bounded, simply-connected
			Lipschitz polygon. 
			Every $\tau\in\Sigma$ admits a unique decomposition
			$\tau=pI+\nabla^\perp_{\mathbb S}q$,
			with $p\in H_0^1(\Omega)$ and 
			$q\in H^1(\Omega;\mathbb R^2)/\mathit{RT}$.
			Moreover, $p\in H^{3/2+\varepsilon}(\Omega)$ for 
                        the
			$0<\varepsilon \leq 1/2$ describing the elliptic
                        regularity of the Laplacian on the polygon $\Omega$.
			The functions $p$ and $q$ depend linearly on $\tau$ and satisfy
			\[
			\|p\|_{H^{3/2+\varepsilon}(\Omega)}
			\lesssim \|\dDiv\tau\|_{L^2(\Omega)},
			\qquad
			\|q\|_{H^1(\Omega)}\lesssim\|\tau\|_\Sigma.
			\]
			If additionally $\tau\in H^\alpha(\Omega;\mathbb S)$ for some
			$0\le\alpha\le3/2+\varepsilon$, then
			$q \in H^{1+\alpha}(\Omega;\mathbb R^2)$ and it
			satisfies
			\begin{equation}
				\label{eq:decomposition}
				\|q\|_{H^{1+\alpha}(\Omega)}
				\lesssim
				\|\tau\|_{H^\alpha(\Omega)}
				+\|\dDiv\tau\|_{L^2(\Omega)}.
			\end{equation}
		   If $\tau$ additionally belongs to $H^s(\Omega;\mathbb{S})$
		   with $s>1/2$ or $\tau$ is
		   piecewise polynomial, then
			$q|_{\partial \Omega}\in H^1(\partial \Omega;\mathbb R^2)$.
		\end{lemma}
		
		\begin{proof}
		    The proof follows ideas from \cite[Theorem 3.1]{KrendlRafetsederZulehner2016}.
			Let $p\in H_0^1(\Omega)$ denote the solution to the Poisson
			problem
			$\Delta p=\dDiv\tau$. Elliptic regularity on polygons 
			\cite{Grisvard1985} shows, for some $0<\varepsilon \leq 1/2$,
			that
			$
			\|p\|_{H^{3/2+\varepsilon}(\Omega)} 
			\lesssim\lVert\dDiv\tau\rVert_{L^2(\Omega)}
			$. 
			Furthermore,
			$\eta:=\tau-pI$ satisfies $\dDiv\eta=0$. 
			If $\eta\in H^\alpha(\Omega;\mathbb S)$,
			\cite[Theorem 4.9]{CostabelBogovskii2010} on the simply-connected
			domain $\Omega$ shows the existence
			$a\in H^\alpha(\Omega)$ with
			$\nabla^\perp a=\Div\eta$ and
			$\|a\|_{H^\alpha(\Omega)}\lesssim\|\eta\|_{H^\alpha(\Omega)}$.
			For $J=(0,-1;1,0)$, the rows of
			$\eta+aJ$ are divergence-free. 
                        Invoking the same result
			row-wise and taking symmetric parts yields a potential
			$q\in H^{1+\alpha}(\Omega;\mathbb R^2)$ satisfying
			\[
			\nabla^\perp_{\mathbb S}q=\eta, \qquad
			\|q\|_{H^{1+\alpha}(\Omega)} \lesssim\|\eta\|_{H^\alpha(\Omega)}.
			\]
			The field $q$ is unique among the choices orthogonal to 
			$\mathit{RT}$.
			Taking $\alpha=0$ proves the claimed decomposition
			and the asserted $H^1$ bound.
			For the proof of uniqueness of the decomposition, it suffices
			to consider the zero decomposition $pI+\nabla^\perp_{\mathbb S}q=0$.
			Taking the double divergence gives $\Delta p=0$,
			and the unique solvability of the Dirichlet-Laplacian shows uniqueness of
			$p$.
			The uniqueness of $q$ follows from orthogonality to $\mathit{RT}$.
			Finally, if $\tau\in H^\alpha(\Omega;\mathbb S)$ with
			$0\le\alpha\le3/2+\varepsilon$, then
			\[
			\|\eta\|_{H^\alpha(\Omega)}
			\lesssim\|\tau\|_{H^\alpha(\Omega)}
			+\|\dDiv\tau\|_{L^2(\Omega)}.
			\]
			The preceding potential construction and uniqueness give
			\eqref{eq:decomposition} for the same $q$.
			
            For the proof of the boundary regularity, let $\tau$ be 
            $H^s$ regular with $s>1/2$.
			It was already shown that
			$q\in H^{1+\beta}(\Omega;\mathbb R^2)$
			for $\beta=\min\{s,1/2+\varepsilon\}$,
			thus $q|_{\partial \Omega}\in H^1(\partial \Omega;\mathbb R^2)$. 
			If $\tau$ is piecewise polynomial,
			$\eta=\nabla^\perp_{\mathbb S}q=\tau-pI$ in the above construction
			satisfies
			$\eta|_K\in H^{3/2+\varepsilon}(K;\mathbb S)$ on every element $K$.
			Taking derivatives in the components of the matrix
			$\eta=\nabla^\perp_{\mathbb S}q$
			gives the (distributional) relations
			for any pair $j\neq k$ from $\{1,2\}$ as follows
            \[
				\partial_{jj}q_j
				=(-1)^j ( 2\partial_j\eta_{jk}+\partial_k\eta_{kk}),
				  \partial_{jk}q_j =(-1)^k\partial_j\eta_{jj},
				\partial_{jj}q_k  =(-1)^j\partial_j\eta_{kk}.
			\]		
			Thus $q|_K\in H^{5/2+\varepsilon}(K;\mathbb R^2)$.
			The embedding $H^{5/2+\varepsilon}(K)\hookrightarrow C^0(\overline K)$
			provides continuous representatives on each closed element.
			Since $q\in H^1(\Omega;\mathbb R^2)$, the traces of adjacent
			elements coincide almost everywhere on their common edge.
			By continuity, they coincide on the whole edge, including
			its endpoints. Propagating this equality through each vertex
			patch shows that the boundary edge traces have matching
			endpoint values, which, together with the edge-wise
			$H^1$ regularity, gives
			$H^1(\partial \Omega;\mathbb R^2)$.
		\end{proof}
	
	\begin{remark}
	\label{remark:svarepsilon}
        Since $\varepsilon =1/2$ on convex domains and
        $s<1$ on non-convex polygons \cite[Chapter~3]{Grisvard1992}, 
        we always have $s\leq 3/2+\varepsilon$
        for $\varepsilon$ from Lemma~\ref{lem:decomposition}.
	\end{remark}
	
For $g=(g_1,g_2)\in X$ and $\tau\in\Sigma^r+\Sigma_h$,
let $\tau=pI+\nabla^\perp_{\mathbb S}q$ be the unique decomposition
in Lemma~\ref{lem:decomposition}. Since $p|_{\partial \Omega}=0$, define
\begin{equation}
	\label{def:divdiv-trace}
	\langle g,\tau\rangle
	:=(\partial_np,g_1)_{L^2(\partial \Omega)}
	-(\partial_tq,t\,\partial_tg_1+n\,g_2)_{L^2(\partial \Omega)}.
\end{equation}
Here $\partial_tq$ denotes the tangential derivative of
$q|_{\partial \Omega}$. The regularity of $p$ and the
last statement of Lemma~\ref{lem:decomposition} together with
the linearity ensure that both terms are
well defined for any $\tau\in\Sigma^r+\Sigma_h$.
Moreover, the decomposition estimates and
the trace theorem give $|\langle g,\tau\rangle| \lesssim\|g\|_X\|\tau\|_{\Sigma^r}$ for all $\tau\in\Sigma^r$,
so $g$ induces a continuous functional on $\Sigma^r$.

\begin{remark}
\label{rem:consistency}
	The boundary functional \eqref{def:divdiv-trace} agrees with
	the trace defined via Green's identity as
$$
 \langle g,\tau\rangle =(\dDiv\tau,G)_{L^2(\Omega)} -(\tau, \nabla^2G)_{L^2(\Omega)}
  ,$$	
	whenever $g$ admits an
	$H^2(\Omega)$ lifting $G$.
	Indeed, assume that $g_1=G|_{\partial \Omega}$ and $g_2=\partial_nG$
    for some $G\in H^2(\Omega)$. Since $\Delta p=\dDiv\tau$
and $p|_{\partial \Omega}=0$, Green's formula gives
$(\Delta p,G)_{L^2(\Omega)} -(p,\Delta G)_{L^2(\Omega)} =(\partial_np,g_1)_{L^2(\partial \Omega)}.$
Since the rows of $\nabla^\perp q$ are divergence-free
with normal trace $\partial_tq$, it holds that $
(\nabla^\perp_{\mathbb S}q,\nabla^2G)_{L^2(\Omega)}
=(\partial_tq,\nabla G)_{L^2(\partial \Omega)}
=(\partial_tq,t\,\partial_tg_1+n\,g_2)_{L^2(\partial \Omega)}$.
This implies that the definition in \eqref{def:divdiv-trace} is consistent 
with the usual $\dDiv$ traces for regular boundary data. 
\end{remark}

\begin{remark}
\label{remark:discrete-consistency}
For $\tau_h\in\Sigma_h$, the boundary functional \eqref{def:divdiv-trace} coincides with the boundary functional used in the discrete formulation of \cite{FuhrerHeuer2025}. Indeed, this follows from Remark~\ref{rem:consistency}, the element-wise Green formula and density and continuity of both boundary functionals with respect to the norm of $X$. This result provides an extension of its boundary action to the test space required by the transposition formulation.
\end{remark}
	
	We next estimate the discrete boundary functional.
	Let $\tau_h=pI+\nabla^\perp_{\mathbb S}q$ be the unique
	decomposition of $\tau_h\in\Sigma_h$
        from Lemma~\ref{lem:decomposition}.
	Since $\tau_h$ is piecewise polynomial and thus
	$H^\alpha$-regular for any $\alpha<1/2$,
	the $q$ component has global
	$H^{1+\alpha}$ regularity.
	The triangle inequality gives
	\begin{equation}
		\label{ds-biharmonic-0}
		| \langle g, \tau_h \rangle| \leq | (\partial_np,g_1)_{L^2(\partial \Omega)}|
		+ | (\partial_tq,\, t\,\partial_tg_1)_{L^2(\partial \Omega)}| +  | (\partial_tq,\, n\,g_2)_{L^2(\partial \Omega)}|.
	\end{equation}
	The first term can be bounded by $|(\partial_np,g_1)_{L^2(\partial\Omega)}|
	\lesssim
	\|\operatorname{divDiv}\tau_h\|_{L^2(\Omega)}
	\|g_1\|_{L^2(\partial\Omega)}$. For the second term with $t_1=1$, it follows from the scaled trace inequality that
	\begin{equation}
		\label{ds-biharmonic:1}
		\begin{aligned}
			|(\partial_tq,t\,\partial_tg_1)_{L^2(\partial \Omega)}|  & \le  
			\|\partial_tq\|_{L^2(\partial \Omega)} \|\partial_tg_1\|_{L^2(\partial \Omega)} \\
			&\lesssim 
			\|g_1\|_{H^{1}(\partial \Omega)} (  h^{-1/2}\|q\|_{H^1(\Omega)}  + h^{1/2}\|\nabla_h^2q\|_{L^2(\Omega)})
			.
		\end{aligned}
	\end{equation}
	The local differential identities in the proof of Lemma~\ref{lem:decomposition}
	and the polynomial inverse inequality imply
	$\| \nabla_h^2q\|_{L^2(\Omega)} \lesssim \|\nabla_h(\tau_h-pI)\|_{L^2(\Omega)} \lesssim  h^{-1}\|\tau_h\|_{L^2(\Omega)} +\|p\|_{H^1(\Omega)}.$ This,  \eqref{ds-biharmonic:1} and \eqref{eq:decomposition} give $|(\partial_tq,t\,\partial_tg_1)_{L^2(\partial \Omega)}|  \lesssim h^{-1/2} \|g_1\|_{H^{1}(\partial \Omega)} \| \tau_h\|_\Sigma$. 
	For the second term of \eqref{ds-biharmonic-0} 
	with $1<t_1 < 3/2$, it follows from the trace theorem that 
	$$
	\begin{aligned}
		& |(\partial_tq,t\,\partial_tg_1)_{L^2(\partial \Omega)}|  \lesssim
		\|\partial_tq\|_{H^{1-t_1}(\partial \Omega)} \|\partial_tg_1\|_{H^{t_1-1}(\partial \Omega)}  \lesssim 	\|q\|_{H^{5/2-t_1}(\Omega)} \|g_1\|_{H^{t_1}(\partial \Omega)} \\
		&\quad  \lesssim (\|\tau_h\|_{H^{3/2-t_1}(\Omega)} + \|\dDiv \tau_h\|_{L^2(\Omega)} ) \|g_1\|_{H^{t_1}(\partial \Omega)} \lesssim
		h^{t_1-3/2}  \|g_1\|_{H^{t_1}(\partial \Omega)} \|\tau_h\|_\Sigma
        ,
	\end{aligned}
	$$
	where the last two estimates follow by \eqref{eq:decomposition} and the
	inverse estimate for $\tau_h$. 
	Similar considerations for the third term of \eqref{ds-biharmonic-0} give
	$$
	\begin{aligned}
		| (\partial_tq,\, n\,g_2)_{L^2(\partial \Omega)}| & \lesssim  h^{-1/2} \|g_2\|_{L^2(\partial \Omega)} \| \tau_h\|_\Sigma  &&\text{ for }  t_2 =0, \\
		| (\partial_tq,\, n\,g_2)_{L^2(\partial \Omega)}| &\lesssim  h^{t_2-1/2} \|g_2\|_{H^{t_2}(\partial \Omega)} \| \tau_h\|_\Sigma  &&\text{ for }  0< t_2 < 1/2.
	\end{aligned}
	$$ 
	Substituting these estimates into \eqref{ds-biharmonic-0} gives the discrete stability 
	$$
	\begin{aligned}
		\|\sigma_h\|_\Sigma+\|u_h\|_{L^2(\Omega)} \lesssim
		\|f\|_{L^2(\Omega)} + h^{t_1-3/2}\|g_1\|_{H^{t_1}(\partial \Omega)} +
		h^{t_2-1/2}\|g_2\|_{H^{t_2}(\partial \Omega)} .
	\end{aligned}
	$$

	Let $P_h:L^2(\Omega)\to U_h$ denote the $L^2$-projection,
	which satisfies
	$$
	\|w-P_h w\|_{L^2(\Omega)} \le C h^2 \|w\|_{U^r} 
	\quad \text{ for all } w \in U^r
	$$
	so that Assumption \ref{condition1} holds with $\rho_U=h^2$.
	Let $\Pi_h^{\mathrm{FH}}$ denote the 
	Führer–Heuer interpolant in \cite[Section 3.3]{FuhrerHeuer2025}. Let $I_h^2$ be the continuous quadratic Lagrange interpolant for vector fields. For $\phi \in\Sigma^r$, let $\phi=pI+\nabla^\perp_{\mathbb S}q$ be its unique decomposition and define 
	\begin{equation}
		\label{def:Pih}
		\Pi_h\phi = \Pi_h^{\mathrm{FH}}(pI) + \nabla^\perp_{\mathbb S}(I_h^2q),\quad \delta_p=pI-\Pi_h^{\mathrm{FH}}(pI), \quad \delta_q=q-I_h^2q.
	\end{equation}
	By \cite[Proposition 10]{FuhrerHeuer2025}, it holds that
	\begin{equation}
		\operatorname{divDiv}\Pi_h^{\mathrm{FH}} (pI)
		= P_h \Delta p, \quad
		\|pI -\Pi_h^{\mathrm{FH}}(pI)\|_{L^2(\Omega)} \lesssim h^{3/2+\varepsilon}\|p\|_{H^{3/2+\varepsilon}(\Omega)}.
	\end{equation}
	The nodal values of $q$ are well defined because $q\in H^{1+s}(\Omega;\mathbb R^2)$, with $s>1/2$.
	Since $P_2(K;\mathbb R^2)\subset RT_2(K)$, $
	\nabla^\perp_{\mathbb S}(I_h^2q)|_K$ belongs to the local shape-function space.
	Moreover, $I_h^2q\in H^1(\Omega;\mathbb R^2)$, so its symmetric curl has vanishing distributional double divergence. Thus $\Pi_h\phi \in\Sigma_h$, and $ \operatorname{divDiv}\Pi_h\phi=
	P_h\Delta p = P_h\operatorname{divDiv}\phi$. Assumption~\ref{condition2}(i) holds. The approximation estimates for the two components give
	$$
	\begin{aligned}
		\|\phi-\Pi_h\phi\|_{L^2(\Omega)} \lesssim h^{3/2+\varepsilon}\|p\|_{H^{3/2+\varepsilon}(\Omega)}+h^s\|q\|_{H^{1+s}(\Omega)} \lesssim h^s\|\phi\|_{\Sigma^r}
         ,
	\end{aligned}
	$$
        where we used Remark~\ref{remark:svarepsilon}.
	Thus $\Pi_h:\Sigma^r\to\Sigma_h$ is a continuous linear commuting approximation operator. This gives Assumption~\ref{condition2}(ii) with $\rho_\Sigma = h^s$.   For $g \in X$, since $p|_{\partial\Omega}=0$, preservation of the normal-normal edge moments and $(\Pi_h^{\mathrm{FH}}(pI))_{nn}|_e\in P_1(e)$ imply
	$(\Pi_h^{\mathrm{FH}}(pI))_{nn}=0$ on every boundary edge.
	Consequently, $\langle(0,g_2),\delta_p\rangle=0$ and
	$$
		\langle g, \phi-\Pi_h\phi \rangle 
		=\langle g, \delta_p \rangle + \langle g, \nabla^\perp_{\mathbb S}\delta_q \rangle 
		=\langle \left(\begin{smallmatrix}g_1\\0\end{smallmatrix}\right),
		                                                   \delta_p \rangle 
		 +\langle \left(\begin{smallmatrix}g_1\\0\end{smallmatrix}\right), 
		                           \nabla^\perp_{\mathbb S}\delta_q \rangle  
        + \langle \left(\begin{smallmatrix}0\\g_2\end{smallmatrix}\right),
                                   \nabla^\perp_{\mathbb S}\delta_q \rangle
		.
	$$    
 By Remark~\ref{remark:discrete-consistency}, the boundary functional defined 
 in \eqref{def:divdiv-trace} agrees on $\Sigma_h$ with the boundary action 
of \cite{FuhrerHeuer2025}. The interpolant $\Pi_h^{\mathrm{FH}}$
 preserves the \(P_1(e)\)-moments of the normal-normal and effective-shear 
traces as well as the corner-force degrees of freedom
 \cite[eqs.~(5)--(6)]{FuhrerHeuer2025}. 
For \(pI\), the normal-normal trace equals \(p\), the 
effective-shear trace equals \(\partial_np\), and all corner-force 
contributions vanish. Since \(p|_{\partial\Omega}=0\), it follows that
	$$
	\langle\left(\begin{smallmatrix}g_1\\0\end{smallmatrix}\right)
              ,\delta_p\rangle
	=
	\sum_{e \in\mathcal E_h^\partial}
	\bigl(
	(1-Q_e^1)\partial_np,\,g_1
	\bigr)_{L^2(e)} \leq \sum_{e\in\mathcal E_h^\partial}
	\|(1-Q_e^1)\partial_np\|_{L^2(e)}
	\|(1-Q_e^1)g_1\|_{L^2(e)}
	.
	$$
	For $e \in \mathcal{E}(K)$, subtracting an affine approximation of $p$, applying the scaled trace inequality, and using the approximation property of $Q_e^1$ yield $ \|(1-Q_e^1)\partial_np\|_{L^2(e)}
	\lesssim h_K^{\varepsilon}|p|_{H^{3/2+\varepsilon}(K)}$. Indeed, the affine approximation has a constant normal derivative, while its restriction to $e$ belongs to $P_1(e)$; both are annihilated by $1-Q_e^1$.
	This, the edge projection estimates for the data and summing over boundary edges gives
	\begin{equation}
		\label{eq:41}
		|\langle(g_1,0),\delta_p\rangle| \lesssim h^{t_1+\varepsilon} \|g_1\|_{H^{t_1}(\partial \Omega)} \|p\|_{H^{3/2+\varepsilon}(\Omega)}
          .
	\end{equation}
	By the global boundary identity,
	$$
	\langle(g_1,0),\nabla^\perp_{\mathbb S}\delta_q\rangle
	=
	-(\partial_t\delta_q,t\,\partial_tg_1)_{L^2(\partial \Omega)}.
	$$
	For $t_1=1$, the Cauchy--Schwarz inequality and the derivative trace estimate for quadratic Lagrange approximation directly give
	$$
	|\langle(g_1,0),\nabla^\perp_{\mathbb S}\delta_q\rangle|
	\lesssim \|  \partial_t\delta_q\|_{L^2(\partial \Omega)} \| t\,\partial_tg_1\|_{L^2(\partial \Omega)}
	\lesssim h^{s-1/2}
	\|g_1\|_{H^1(\partial \Omega)}
	\|q\|_{H^{1+s}(\Omega)}.
	$$
	For $1<t_1<3/2$, note that  $I_h^2q$ preserves the values at both endpoints of every mesh edge, then $\int_e \partial_t\delta_q\,ds=0$ for every $e \in\mathcal E_h^\partial$. This and $t$ is constant on $e$ give $
	(\partial_t\delta_q,t\,\partial_tg_1)_{L^2(e)} = (\partial_t\delta_q,\,
	t(\partial_tg_1-Q_e^0\partial_tg_1) )_{L^2(e)}$. 
	Therefore, the Cauchy--Schwarz inequality and the 
         fractional Poincar\'e estimate on each edge yield
	\begin{equation}
		\label{eq:42}
		|\langle(g_1,0),\nabla^\perp_{\mathbb S}\delta_q\rangle| \lesssim h^{s+t_1-3/2} \|g_1\|_{H^{t_1}(\partial \Omega)} \|q\|_{H^{1+s}(\Omega)}.
	\end{equation}
	The discussion of $t_1=1$ is covered also by \eqref{eq:42}. Following similar procedures as above, we can verify the estimate of the term
	$$
	|\langle(0,g_2),\nabla^\perp_{\mathbb S}\delta_q\rangle|=
	|(\partial_t\delta_q,ng_2)_{L^2(\partial \Omega)}| \lesssim h^{s+t_2-1/2}
	\|g_2\|_{H^{t_2}(\partial \Omega)}
	\|q\|_{H^{1+s}(\Omega)}.
	$$
	Thus,  for every $\phi \in\Sigma^r$ and $g\in X$,
	$$
	\begin{aligned}
		|\langle g,\phi-\Pi_h\phi\rangle|
		\lesssim \Big(
		h^{s+t_1-3/2}\|g_1\|_{H^{t_1}(\partial \Omega)} +
		h^{s+t_2-1/2}\|g_2\|_{H^{t_2}(\partial \Omega)}
		\Big)\|\phi\|_{\Sigma^r}.
	\end{aligned}
	$$
	Assumption~\ref{condition2}(iii) holds with $\rho_X = h^{\min\{s+t_1-3/2, s+t_2-1/2\}}  $. Since $\Sigma_h \subset \Sigma$, $\rho_{\mathrm{nc}}=0$.
	
	\begin{corollary}
		\label{cor:biharmonic}
		Let $\Omega\subset\mathbb R^2$ be a bounded, simply connected
		polygonal Lipschitz domain, and let $1/2<s\le 2$
		be an admissible regularity index for the homogeneous clamped
		biharmonic problem. For $f\in L^2(\Omega)$ and
		$g=(g_1,g_2)\in
		X=H^{t_1}(\partial\Omega)\times H^{t_2}(\partial\Omega)$ with $1\le t_1<3/2$ and $0\le t_2<1/2$.
		Let $u$ be the transposition solution of the biharmonic equation \eqref{eq:biharmonic} and let
		$(\sigma_h,u_h)\in\Sigma_h\times U_h$ solve the mixed method.
		Then it holds that
		\[
		\|u-u_h\|_{L^2(\Omega)}
		\lesssim
		h^s\|f\|_{L^2(\Omega)}
		+h^{s+t_1-3/2}\|g_1\|_{H^{t_1}(\partial\Omega)}
		+h^{s+t_2-1/2}\|g_2\|_{H^{t_2}(\partial\Omega)}.
		\]
	\end{corollary}

	\begin{example}
		\label{ex:biharmonic}
		We consider the biharmonic problem on the square and L-shaped
                domains from Example~\ref{ex:laplace} and the exact solution
		$$
		u(r,\theta)=r^{1+\alpha}\sin\bigl((1+\alpha)\theta\bigr),\qquad \alpha=-1/3 .
		$$
		It satisfies $f=\Delta^2u=0$.  The boundary terms satisfy
		$$
		g_1\in H^{7/6-\nu}(\partial\Omega),\qquad g_2\in H^{1/6-\nu}(\partial\Omega)\quad  \text{ for all } \nu>0.
		$$ 
		
		On the square domain, the regularity-shift index of the adjoint clamped-plate
		problem $s_{S}=2$, and hence $\|u-u_h\|_{L^2(\Omega)}$ behaves
		like $h^{5/3-\nu}$.  For the standard L-shaped domain, the first clamped-plate singular exponent gives $s_{L}=0.544\ldots$, so that the predicted rate is $s_{L}+\alpha = 0.211\ldots$. Table \ref{tab:biharmonic_fh} shows $L^2$ errors and convergence rates of the lowest-order element, which coincide with the theoretical results 
of Corollary \ref{cor:biharmonic}.
		\qed
		
\begin{table}[tbp]
	\centering\small
	\begin{tabular}{ccccc}
		\toprule
		& \multicolumn{2}{c}{Unit square}
		& \multicolumn{2}{c}{L-shaped domain}\\
		\cmidrule(lr){2-3}\cmidrule(lr){4-5}
		$h_\ell/\sqrt{2}$
		& $\|u-u_h\|_{L^2(\Omega)}$ & order
		& $\|u-u_h\|_{L^2(\Omega)}$ & order\\
		\midrule
		$2^{-1}$ & $5.3871\mathrm{e}{-3}$ & --
		& $1.0866\mathrm{e}{-2}$ & --\\
		$2^{-2}$ & $1.8247\mathrm{e}{-3}$ & $1.5618$
		& $7.2860\mathrm{e}{-3}$ & $0.5766$\\
		$2^{-3}$ & $5.9664\mathrm{e}{-4}$ & $1.6127$
		& $6.2021\mathrm{e}{-3}$ & $0.2324$\\
		$2^{-4}$ & $1.9167\mathrm{e}{-4}$ & $1.6382$
		& $5.4977\mathrm{e}{-3}$ & $0.1739$\\
		$2^{-5}$ & $6.1105\mathrm{e}{-5}$ & $1.6493$
		& $4.8311\mathrm{e}{-3}$ & $0.1865$\\
		$2^{-6}$ & $1.9393\mathrm{e}{-5}$ & $1.6557$
		& $4.2090\mathrm{e}{-3}$ & $0.1989$\\
		$2^{-7}$ & $6.1370\mathrm{e}{-6}$ & $1.6600$
		& $3.6502\mathrm{e}{-3}$ & $0.2055$\\
		\bottomrule
	\end{tabular}
	\caption{$L^2$ errors and observed convergence orders for the
		lowest-order F\"uhrer--Heuer element approximation of the
		two-dimensional biharmonic problem with the boundary data
		$g_1 \in H^{7/6-\nu}$ and
		$g_2 \in H^{1/6-\nu}$.}
	\label{tab:biharmonic_fh}
\end{table}
		
	\end{example}

	\section{Application to Maxwell equation}
	\label{sec:Maxwell}
	
	Let $\Omega\subset\mathbb R^3$ be a bounded simply connected Lipschitz polyhedron with connected boundary $\partial \Omega$.
	We consider the curl-curl problem
	$$
    \curl \curl  u = f,  \quad  \operatorname{div} u =0  \text{ in } \Omega, \quad  n\times u = g  \text{ on } \partial \Omega.
	$$
	We assume the compatibility conditions $\ddiv f =0$ in $\Omega$ and $g\cdot n=0$ on $\partial \Omega$.
	The mixed formulation is based on the auxiliary variable $\sigma = \curl u $
    and the spaces
	$$
	\begin{aligned}
		\Sigma &= H(\curl, \Omega; \mathbb{R}^3) := \{  \tau \in L^2(\Omega; \mathbb{R}^3): \curl \tau  \in L^2(\Omega; \mathbb{R}^3)  \}, \\
		U &= H(\ddiv^0,\Omega;\mathbb{R}^3):=\{ v \in  H(\ddiv,\Omega; \mathbb{R}^3): \ddiv v= 0 \text{ in } \Omega \},
	\end{aligned}
	$$
	equipped with the norms $\| \cdot \|_{\Sigma} = \| \cdot\|_{L^2(\Omega)} + \|\curl \cdot \|_{L^2(\Omega)}$ and $\| \cdot \|_U = \| \cdot \|_{L^2(\Omega)}$, respectively. 
	We consider the bilinear forms 
	$$
	a(\sigma, \tau) = ( \sigma, \tau)_{L^2(\Omega)}, \quad 
	b(\tau, v) = -(\curl \tau, v )_{L^2(\Omega)}  \text{ for all } \sigma,\tau \in \Sigma, v \in U.
	$$
	As before, the form $a$ extends to $H=L^2(\Omega;\mathbb R^3)$ equipped with the
    $L^2$ norm.
	Assumption \ref{ass:saddle} follows from the mixed formulation
        of the curl-curl problem \cite{Monk2003}.
    Let $s \in (1/2,1]$ be an admissible elliptic shift exponent so that
    the solution to the homogeneous Dirichlet problem of the curl-curl
    equation is $H^s(\Omega)$ regular 	\cite{CostabelDauge2000}.
    By $H_0(\curl, \Omega; \mathbb{R}^3)$ we denote the closure of
    smooth and compactly supported vector fields with respect to the 
    norm $\|\cdot\|_\Sigma$.
    Let $U^r = H^s(\Omega;\mathbb{R}^3) \cap H_0(\curl,\Omega;\mathbb R^3) \cap U$ with the norm $\| \cdot \|_{U^r} = \| \cdot\|_{H^s(\Omega)} + \| \curl  \cdot \|_{L^2(\Omega)}$ and 
	$$
	\Sigma^r = \Sigma \cap H^{s}(\Omega; \mathbb{R}^3) \text{ with the norm } \|\cdot \|_{\Sigma^r} = \|\cdot \|_{H^{s}(\Omega)} + \|\operatorname{curl}\cdot \|_{L^2(\Omega)}.
	$$  
	By \cite{CostabelDauge2000},
    for any $z\in U$, the solution to $T(\phi,w)=(0,z)$ is unique and 
	satisfies
	$$
	\| w \|_{H^s(\Omega)} + \|  \phi \|_{H^s(\Omega)} \lesssim \| z\|_{L^2(\Omega)}.
	$$
	This verifies Assumption \ref{ass:regularity}.
	
	Let $\mathcal T_h$ be a tetrahedral triangulation of $\Omega$
    as in prior sections.
    Given $K \in\mathcal T_h$, we define the usual N\'ed\'elec space
	as $N_0(K):= \{\tau:K\to\mathbb R^3 : \tau(x) = \alpha+\beta\times x
                   \text{ for }\alpha,\beta\in\mathbb R^3\}$.
	We consider the conforming lowest-order discrete de Rham pair in \cite{Boffi2017}
	$$
	\Sigma_h = \{\tau\in\Sigma: 
                    \tau|_K\in N_0(K)\, \text{ for all } K\in\mathcal T_h \}
                    , \qquad U_h=\curl\Sigma_h.
	$$
	By exactness of the discrete de Rham complex, 
    $U_h$ equals the divergence-free lowest-order Raviart--Thomas
    space.
	The corresponding discrete mixed problem is uniformly stable
    \cite{Monk2003}, which gives Assumption \ref{ass:discrete-stability}.
    For any $\tau \in \Sigma$, let $\gamma_t \tau:=(n\times(\tau\times n))|_{\partial \Omega}$ denote its tangential component. For $0\le t<1/2$, define $X =\{ g \in H^t(\partial \Omega; \mathbb{R}^3): g \cdot n = 0 \text{ on } \partial \Omega\}$ with $\|g\|_X := \|g\|_{H^t(\partial \Omega)}$.

\begin{remark}
We consider the Maxwell problem with $H^t$ boundary data,
but unlike in the other PDE examples of this paper, this does not
describe a superspace of the trace space of $\Sigma$,
which is described in \cite{BuffaCostabelSheen2002}.
Still, boundary data in $H^t$ can admit functions that need not
belong to the the standard Maxwell trace space.
\end{remark}

		For $\tau_h\in\Sigma_h$, define 
        $\langle g,\tau_h\rangle =  (g,\gamma_t\tau_h)_{L^2(\partial \Omega)}$.
      For $t=0$, the trace and inverse inequalities show
		\begin{equation}
			\label{eq: max-t=0}
			|\langle g, \tau_h \rangle|
           \leq \| g\|_{X} \| \tau_h\|_{L^2(\partial \Omega)}
            \lesssim h^{-1/2}  \| g\|_{X} \| \tau_h\|_{\Sigma}.
		\end{equation}
		For $0<t<1/2$, it holds that $g\times n\in H^t(\partial\Omega;\mathbb R^3)$ and $	\|g\times n\|_{H^t(\partial\Omega)}
		\lesssim\|g\|_X$,
		see \cite[Corollary 1.4.4.5]{Grisvard1985}. 
        By \cite[Theorem 1.5.1.2]{Grisvard1985},
        there exists a lifting $G \in H^{1/2+t}(\Omega;\mathbb R^3)$ such that 
		$G|_{\partial\Omega}=g\times n$ and $\|G\|_{H^{1/2+t}(\Omega)}
		\lesssim\|g\|_X$.
		Since $g\cdot n=0$, we have $n\times G=n\times(g\times n)=g$ 
        and the generalized Green formula
		\begin{equation}
			\label{eq: curl-integration}
			\langle g,\tau_h\rangle =\langle\curl G,\tau_h\rangle_\Omega
			-(G,\curl\tau_h)_{L^2(\Omega)}.
		\end{equation}
		The first term pairs $H^{t-1/2}(\Omega)$ with $H_0^{1/2-t}(\Omega)=H^{1/2-t}(\Omega)$. 
		This and the inverse estimate give
		$|\langle\curl G,\tau_h\rangle_\Omega| \leq \| \curl G\|_{H^{t-1/2}(\Omega)} \| \tau_h\|_{H^{1/2-t}(\Omega)}  \lesssim h^{t-1/2} \| g\|_{X} \| \tau_h\|_{\Sigma}$.
		The second term $|(G,\curl\tau_h)_{L^2(\Omega)}| \leq \| G\|_{L^2( \Omega)} \| \curl \tau_h \|_{L^2(\Omega)} \lesssim \| g\|_X \| \tau_h\|_{\Sigma}$. Substituting these two estimates into \eqref{eq: curl-integration} 
        and combining with \eqref{eq: max-t=0} results in
		\begin{equation}
		\label{eq:max-discrete-stability}
		\|\sigma_h\|_\Sigma+\|u_h\|_U\lesssim\|f\|_{L^2}+h^{t-1/2}\|g\|_X.
		\end{equation}
	
	Let $P_h: U^r \to U_h$ denote the lowest-order Raviart--Thomas interpolant 
	\cite{BoffiBrezziFortin2013}.
    Since $\operatorname{div} w=0$ for $w \in U^r$, 
   the commuting property implies
    $\operatorname{div} P_hw=0$,
	and therefore $P_hw\in U_h$. Hence, 
	$ \|w-P_hw\|_{L^2(\Omega)} \lesssim h^s\|w\|_{H^s(\Omega)}$
    and
    Assumption~\ref{condition1} holds with $\rho_U = h^s$.
Let $\Pi_h: \Sigma^r \to \Sigma_h$ denote the canonical N\'ed\'elec interpolant
   \cite{BoffiBrezziFortin2013}. 
    It does not, in general, satisfy the Fortin property
	$b(\Pi_h\phi,v_h)=b(\phi,v_h)$ for all $\phi\in\Sigma^r$ and 
    $v_h\in U_h$, since the Raviart--Thomas interpolant is not the 
    $L^2$-orthogonal projection onto $U_h$.  However, if 
   $T(\phi, w) = (0, z_h)$ for some $z_h \in U_h$, then $\curl \phi \in U_h$ 
    and $P_h\curl\phi =\curl\phi$. This  implies
	$$
	b(\Pi_h\phi,v_h)=  -(\curl \Pi_h\phi, v_h)_{L^2(\Omega)}
	 =	-(P_h\curl\phi, v_h)_{L^2(\Omega)} 
	 = b(\phi,v_h)
	$$
	 for all $v_h \in U_h$.
	This proves Assumption~\ref{condition2}(i). 
	By \cite{Monk2003,BoffiBrezziFortin2013},
		\begin{equation}
		\label{eq:interpolation-nedelec}
		\|\phi-\Pi_h\phi\|_{L^2(\Omega)} \lesssim
			h^s\|\phi\|_{H^s(\Omega)}	+	h\|\curl\phi\|_{L^2(\Omega)}
                  .
		\end{equation}
		Using the regularity of the dual solution, we obtain
		$$
		\|\phi-\Pi_h\phi\|_{L^2(\Omega)}
		\lesssim h^s\|z_h\|_{L^2(\Omega)}.
		$$ 
		Assumption~\ref{condition2}(ii) holds with $\rho_\Sigma = h^s$. 
		Let $\delta_h :=\phi-\Pi_h\phi$. 
        Since $\curl \phi \in U_h$, it satisfies $\curl \delta_h =0$.
        We next estimate the boundary term $|\langle g,\delta_h\rangle|$.
        Since $\delta_h$ is not a discrete function, we introduce
        a further approximation step.
		Let $q_h$ be a continuous piecewise polynomial quasi-interpolant 
        of $\phi$ satisfying the standard approximation properties. 
        For $t=0$, the triangle inequality, the boundary approximation estimate,
        and the inverse trace inequality applied to $q_h-\Pi_h\phi$ give
		$$
		\begin{aligned}
		\| \delta_h\|_{L^2(\partial \Omega)}
		& \lesssim \| \phi - q_h \|_{L^2(\partial \Omega)}  + h^{-1/2}  \| q_h - \Pi_h \phi \|_{L^2(\Omega)} \\
		& \lesssim h^{s-1/2} \| \phi\|_{H^s(\Omega)}  + h^{-1/2}  \| \delta_h \|_{L^2(\Omega)}.
		\end{aligned}
		$$
		Together with \eqref{eq:interpolation-nedelec}
		and the inverse estimate for the discrete function $\curl\phi$,
        this yields
		\begin{equation}
			\label{eq1:max-t=0}
				|\langle g, \delta_h \rangle|  \leq \| g\|_X \| \delta_h\|_{L^2(\partial \Omega)} \lesssim h^{s-1/2}  \| g\|_X \| \phi\|_{H^s(\Omega)}.
		\end{equation}
		Let now $0<t<1/2$, and let $G$ be the lifting introduced above. 
        By \eqref{eq: curl-integration} and $\curl \delta_h =0$,
         the generalized Green formula simplifies to 
         $\langle g,\delta_h\rangle	= \langle\curl G,\delta_h\rangle_\Omega$.
       Furthermore,
        $$
        \begin{aligned}
        	\| \delta_h\|_{H^{1/2-t}(\Omega)}
        	& \lesssim \| \phi - q_h \|_{H^{1/2-t}(\Omega)}  + h^{t-1/2}  \| q_h - \Pi_h \phi \|_{L^2(\Omega)} \\
        	& \lesssim h^{s+t-1/2} \| \phi\|_{H^s(\Omega)}  + h^{t-1/2}  \| \delta_h \|_{L^2(\Omega)}.
        \end{aligned}
        $$
        This, combined with \eqref{eq:interpolation-nedelec}, gives the estimate
        $$
        |\langle g, \delta_h \rangle| 
         = |\langle \curl G, \delta_h \rangle_\Omega| 
         \leq \| \curl G\|_{H^{t-\frac12}(\Omega)} \| \delta_h\|_{H^{\frac12-t}( \Omega)}
        \lesssim h^{s+t-\frac12} \| g\|_X \|\phi\|_{H^s(\Omega)}.
        $$ 
        This and \eqref{eq1:max-t=0} verify Assumption~\ref{condition2}(iii) with $\rho_X= h^{s+t-1/2}$.
        Since $\Sigma_h \subset \Sigma$, Remark \ref{remark:conforming} applies and $\rho_{\rm nc} = 0$. 
	\begin{corollary}
		\label{cor:maxwell}
		Let $\Omega \subseteq \mathbb{R}^3$ be an open, bounded 
                simply-connected Lipschitz polyhedron with connected boundary
                and let $1/2<s\leq 1$ be an admissible regularity index for 
                the homogenuous Dirichlet curl-curl problem.
                Given $f \in L^2(\Omega; \mathbb{R}^3)$ and
                $g \in H^t(\partial \Omega;\mathbb{R}^3)$ with $0\leq t <1/2$ 
                satisfying the compatibility conditions $\ddiv f =0$ in $\Omega$ 
                and $g\cdot n=0$ on $\partial \Omega$, 
		let $u$ be the transposition solution of the curl-curl problem and let
		$(\sigma_h,u_h)$ solve the discrete mixed problem. Then it holds that
		\begin{equation*}
			\|u-u_h\|_{L^2(\Omega)} \lesssim \inf_{v_h\in U_h} \|u-v_h\|_{L^2(\Omega)} + h^s\|f\|_{L^2(\Omega)} + h^{s+t-1/2}\|g\|_X.
		\end{equation*}
	\end{corollary}

	\begin{remark}
	\label{remark:Maxwell-inf}
		If, in addition, $u\in H^r(\Omega;\mathbb R^3)$ with $r \geq 0$,
                then it can be shown \cite{FalkWinther2014} 
               that
                the best-approximation term on the right-hand
                side of the error estimate in Corollary~\ref{cor:maxwell}
               is controlled by
		$
		h^{\min\{r,1\}} \|u\|_{H^r(\Omega)}.
		$
		Hence
		$$
		\|u-u_h\|_{L^2(\Omega)} \lesssim h^{\min\{r,1\}}\|u\|_{H^r(\Omega)} + h^s\|f\|_{L^2(\Omega)} + h^{s+t-1/2}\|g\|_X.
		$$
	\end{remark}
	
	\begin{example}
		We consider the three-dimensional Maxwell problem on the 
               unit cube $\Omega_C$ and the L-shaped prism $\Omega_P$ from
               Example~\ref{ex:laplace}.
               The exact solution is ${u}(x,y,z) = \bigl(0, 0, r^\alpha \sin(\alpha\theta)\bigr)^T$ with $\alpha = -1/3$ and vanishing source term $ {f} = 0$. 
        The tangential boundary datum $ {g} =  n \times u \in H^{1/6-\nu}(\partial\Omega; \mathbb{R}^3)$ is imposed weakly through the natural duality pairing in the transposition formulation.
		
		Table~\ref{tab:maxwell} reports the $L^2$ errors and the observed convergence orders for the lowest-order N\'ed\'elec edge element approximation. 
          Since $u \in H^{3/2-\nu}(\Omega; \mathbb{R}^3)$, 
          Remark~\ref{remark:Maxwell-inf} shows the first term of the 
          error estimate gives $h^{3/2-\nu}\| u \|_{H^{3/2-\nu}(\Omega)}$.
        On the unit cube, the observed convergence rate is close to the theoretical prediction $s + t - 1/2 = 2/3 - \nu$ (with shift index $s=1$ 
         and boundary exponent $t = 1/6 -\nu$). For the L-shaped prism, 
   the edge singularity limits the regularity index to $s_P = 2/3 -\nu$,
      yielding an expected convergence order of $s_P + t - 1/2 = 1/3-\nu$ 
      for any $\nu$.
		The observed convergence rates are close to this prediction.
		\qed
		
\begin{table}[tbp]
	\centering\small
	\begin{tabular}{ccccc}
		\toprule
		& \multicolumn{2}{c}{Unit cube}
		& \multicolumn{2}{c}{L-shaped prism}\\
		\cmidrule(lr){2-3}\cmidrule(lr){4-5}
		$h_\ell/\sqrt{3}$
		& $\|u-u_h\|_{L^2(\Omega)}$ & order
		& $\|u-u_h\|_{L^2(\Omega)}$ & order\\
		\midrule
		$2^{-1}$ & $8.6655\mathrm{e}{-2}$ & --
		& $2.3272\mathrm{e}{-1}$ & --\\
		$2^{-2}$ & $6.3194\mathrm{e}{-2}$ & $0.4555$
		& $1.7818\mathrm{e}{-1}$ & $0.3853$\\
		$2^{-3}$ & $4.3248\mathrm{e}{-2}$ & $0.5472$
		& $1.3295\mathrm{e}{-1}$ & $0.4224$\\
		$2^{-4}$ & $2.8424\mathrm{e}{-2}$ & $0.6055$
		& $9.8980\mathrm{e}{-2}$ & $0.4257$\\
		$2^{-5}$ & $1.8306\mathrm{e}{-2}$ & $0.6348$
		& $7.4432\mathrm{e}{-2}$ & $0.4112$\\
		$2^{-6}$ & $1.1670\mathrm{e}{-2}$ & $0.6495$
		& $5.6699\mathrm{e}{-2}$ & $0.3926$\\
		\bottomrule
	\end{tabular}
	\caption{$L^2$ errors and observed convergence orders for the
		lowest-order N\'ed\'elec approximation of the three-dimensional
		Maxwell problem with boundary data
		$g \in H^{1/6-\nu}(\partial\Omega)$.}
	\label{tab:maxwell}
\end{table}

	\end{example}

	%-----------------------------------------------------------------------
	\bibliographystyle{abbrv}
	\bibliography{references-transposition}
	%-----------------------------------------------------------------------

\end{document}